\documentclass{svjour3}                     
\smartqed  
\usepackage{graphicx}
\usepackage[T1]{fontenc}

\usepackage{mathdots}
\usepackage{epsfig}
\usepackage{color}
\usepackage{array}
\usepackage{anysize}
\usepackage[a4paper, total={5.5in, 8in}]{geometry}
\usepackage{latexsym,amssymb,amsmath}
\usepackage{mathptmx}      
\usepackage{hyperref}
\def\({\left(}
\def\){\right)}

\def\trace{\mathop{\mathrm{trace}}}

\def\diag{\mathop{\mathrm{diag}}}
\def\max{\mathop{\mathrm{max}}}
\def\conj{\mathop{\mathrm{conj}}}
\def\cond{\mathop{\mathrm{cond}}}
\def\argmin{\mathop{\mathrm{argmin}}}

\numberwithin{equation}{section}
\numberwithin{theorem}{section}
\numberwithin{lemma}{section}
\numberwithin{corollary}{section}
\numberwithin{proposition}{section}
\numberwithin{example}{section}
\spnewtheorem{algorithm}{Algorithm}[section]{\bf}{\rm}
\journalname{}
\begin{document}

\title{Computing the Mittag-Leffler Function of a Matrix Argument 
}


\author{Jo\~{a}o R. Cardoso}


\institute{Jo\~{a}o R. Cardoso \at Polytechnic Institute of Coimbra, Coimbra Institute of
Engineering, \\
Rua Pedro Nunes - Quinta da Nora, 3030-199\\
Coimbra, Portugal
           and\\
              CMUC -- Center for Mathematics, University of Coimbra, \\
              Coimbra, Portugal \\
              \email{jocar@isec.pt}   }        

\date{Received: date / Accepted: date}

\maketitle

\begin{abstract}

It is well-known that the two-parameter Mittag-Leffler (ML) function plays a key role in Fractional Calculus. In this paper, we address the problem of computing this  function, when its argument is a square matrix. Effective  methods for solving this problem involve the computation of higher order derivatives or require the use of mixed precision arithmetic. In this paper, we provide an alternative method that is derivative-free and works entirely using IEEE standard double precision arithmetic. If certain conditions are satisfied, our method uses a Taylor series representation for the ML function; if not, it switches to a Schur-Parlett technique that will be combined with the Cauchy integral formula. A detailed discussion on the choice of a convenient contour is included. Theoretical and numerical issues regarding the performance of the proposed algorithm are discussed. A set of numerical experiments shows that our novel approach is competitive with the state-of-the-art method for IEEE double precision arithmetic, in terms of accuracy and CPU time. For matrices whose Schur decomposition has large blocks with clustered eigenvalues, our method far outperforms the other. Since our method does not require the efficient computation of higher order derivatives, it has the additional advantage of being easily extended to other matrix functions (e.g., special functions).


\end{abstract}

\section{Introduction}

Given a square complex matrix ${A}\in \mathbb{C}^{n\times n}$ and the real numbers $\alpha>0$ and $\beta$, the matrix Mittag-Leffler (ML) function with two parameters can be defined through the convergent power series
\begin{equation}\label{ml-function}
	E_{\alpha,\beta}({A})=\sum_{k=0}^{\infty}\frac{{A}^{k}}{\Gamma(\alpha k+\beta)}=\frac{1}{\Gamma(\beta)}{I}+\frac{{A}}{\Gamma(\alpha +\beta)}+\frac{{A}^{2}}{\Gamma(2 \alpha +\beta)}+\cdots,
\end{equation}
where $\Gamma$ denotes the Euler Gamma function. Note that the ML function is defined for any square matrix  ${A}\in \mathbb{C}^{n\times n}$. For background on matrix functions, see the books \cite{Higham08,Horn94} and the references therein.

The literature on scalar ML functions is vast, so we refer the reader to the books \cite{Gorenflo14,Podlubny1} and the references therein for background on its the theory and applications. With respect to the numerical computation, many techniques have been proposed in the literature. We highlight in particular the state-of-the-art method of \cite{Garrappa15} and also the effective methods of \cite{Gorenflo02} and \cite{Seybold08}. 

In contrast with its scalar counterpart, the ML function with a matrix argument has been less studied, but in the late few years the interest in designing algorithms for its computation has increased \cite{Garrappa11,Garrappa17,Garrappa18,Higham21,Moreti11}. In \cite{Garrappa18}, the authors adapt the Schur-Parlett method of \cite{Davies03} to the particular case of the ML function, which requires the computation of the scalar ML functions $E_{\alpha,\beta}(z)$ and its higher order derivatives $E^{(k)}_{\alpha,\beta}(z)$, with $k=1,2,\ldots.$. At a first glance, this seems to be a trivial task, but the computation of $E^{(k)}_{\alpha,\beta}(z)$ is actually quite challenging. That is why the major part of \cite{Garrappa18} is devoted to finding effective methods for derivative computations. In \cite{Garrappa18}, one can also find a couple of interesting applications of matrix ML functions in Fractional Calculus. More recently, in \cite[Sec. 6]{Higham21}, the authors propose a derivative-free algorithm for the ML matrix function using mixed precision arithmetic. One of the major goals of our paper is to provide a competitive alternative method that is derivative-free and  works entirely in IEEE double precision arithmetic. 

For a matrix $A$ with ``small'' norm, our method evaluates $E_{\alpha,\beta}({A})$ using the Taylor series \eqref{ml-function}; for general matrices, we start by computing the complex blocked  Schur form of $A$ with reordering \cite{Davies03}, $A=UTU^\ast$, where $U$ is unitary and $T$ is triangular with ``well separated'' blocks $T_{11}, \ldots,T_{qq}$, in its diagonal. Each block $T_{ii}$ is triangular with close eigenvalues and is usually known as ``atomic block''. Since 
$$E_{\alpha,\beta}(A)= U\,E_{\alpha,\beta}(T)\,U^\ast,$$
the computation of the ML function of a general square matrix $A$ can be reduced to the computation of the ML function of a triangular matrix $T$. The block diagonal of $E_{\alpha,\beta}(T)$ is $$\diag\left(E_{\alpha,\beta}(T_{11}),\ldots,E_{\alpha,\beta}(T_{qq})\right),$$
with $T$ being viewed as a $q\times q$ block matrix. Once $E_{\alpha,\beta}(T_{ii})$ is evaluated, the remaining entries of $E_{\alpha,\beta}(T)$ may be computed by means of the block Parlett recurrence, which is based on solving Sylvester matrix equations of smaller size; see more details in Section \ref{schur-parlett}. Hence, an important issue now is how to evaluate the ML function of an atomic block in a stable and efficient fashion.  

 An important difference between our method and the one in \cite{Garrappa18} is in the computation of the ML function of the atomic blocks $E_{\alpha,\beta}(T_{ii})$. While \cite{Garrappa18} adapts the method of \cite{Davies03}, which requires the efficient computation of higher order derivatives of the ML function, our proposal uses the well-known Cauchy integral formula for matrix functions. This raises a couple of questions, like: how to find a suitable contour to approximate the Cauchy integral that minimizes the quadrature error and the computational cost? What is the most convenient quadrature for this type of integrals?
 
 Although our method involves a rough estimate of the second-order derivative of a certain function to estimate the length of the radius of a circle (see Section \ref{contour}), this derivative is not used directly in forming $E_{\alpha, \beta}(A)$, making it legitimate to call it derivative-free.
 
 Given a complex square matrix $A$, the matrix version of the Cauchy integral theorem \cite[Ch.1]{Higham08} states that  
 \begin{equation}\label{cauchy}
 	f(A)=\frac{1}{2\pi i}\int_{\mathcal C} f(z)(zI-A)^{-1}\ dz,
 \end{equation}
 where $f$ is assumed to be analytic on and inside a closed contour ${\mathcal C}$ enclosing the spectrum $\sigma(A)$ of $A$. 
 
 Cauchy's integral formula \eqref{cauchy} has rarely been used in the computation of functions of matrices. However, as pointed out in \cite{Trefethen}, an appropriate choice of the contour ${\mathcal C}$ followed by an evaluation by the trapezoidal rule, can lead to very good results for some particular functions and matrices. In this work, we show that, for matrices with clustered eigenvalues, choosing a circle for the contour (with an appropriate center and radius), and applying the trapezoidal rule, 
 provides a promising method for computing the ML function. This technique has the advantage of being easily extendable to other matrix function, namely to functions whose properties are not suitable for exploiting in numerical computations; this is the case for many special functions.
 
 Our interest in the Cauchy integral formula \eqref{cauchy} does not lie specifically in its application for approximating $E_{\alpha,\beta}(A)$ for any square matrix $A$, but instead for computing the ML of the atomic blocks $T_{ii}$ arising in the blocked Schur decomposition of $A$ with reordering. Since these blocks are upper triangular matrices with clustered eigenvalues, the contour circle has a radius that is in general of limited length, thus avoiding high oscillations in the integrand function.  

Although we often use the terminology ``Mittag-Leffler function'' (singular), in fact we are dealing with an infinite family of functions, depending on the values of the parameters $\alpha$ and $\beta$. For instance, if $\alpha=\beta=1$, then $E_{1,1}(z)=e^z$ is the exponential, while $E_{2,2}(z)=\frac{\sinh(\sqrt{z})}{\sqrt{z}}$. Many other known functions can be viewed as particular ML functions \cite{Gorenflo14}. It is easy to observe that the values of $\alpha$ and $\beta$ influence the convergence of the series (\ref{ml-function}) and its practical usage for computational purposes. The Gamma function in the denominators grows very fast for large values of $\alpha$ and slows down for smaller ones (say $0< \alpha< 1$). As we will later, for sufficiently large $\alpha$ and $\beta$ and small$\|A\|$, the Taylor series (\ref{ml-function}) may yield very interesting results in computations.      

The organization of the paper is as follows. In Section \ref{revisiting}, there is a brief revision of the scalar ML function, mainly focused on the existing effective methods for its numerical computation. Section \ref{sec-taylor} provides a novel analysis of the  usage of the Taylor series in numerical computation of the ML function, culminating in a result and a technique that will be incorporated later in the main algorithm for computing the matrix ML function.
In Section \ref{contour}, we discuss which types of contours and quadrature are convenient for making the Cauchy integral formula reliable in numerical computations. In Section  \ref{sec-algorithm}, the main algorithm of the paper is provided in pseudo-code. Numerical considerations on its implementation are included as well. Section \ref{sec-experiments} includes a set of experiments to illustrate the performance of the main algorithm and how it compares with the existing state-of-the-art method for IEEE double precision arithmetic. Finally, in Section \ref{sec-conclusions} some conclusions are drawn.

\section{Revisiting the Scalar ML function}\label{revisiting}

The classical ML function (also called the one-parameter ML function) was introduced and studied by the Swedish mathematician Magnus G\"{o}sta Mittag-Leffler (1846--1927) in a set of five papers \cite{Mittag1,Mittag2,Mittag3,Mittag4,Mittag5} published at the beginning of the 20th century.  The so-called two-parameter ML function, which extends the classical one, was first introduced in the paper \cite{Wiman}, but the author did not pay much attention to it. It was only much later in the papers \cite{Agarwal,Humbert} that the two-parameter ML function was rediscovered and investigated in detail. Since then, the ML functions (with one, two, or even more parameters) have attracted the interest of many researchers. This is in part due to the relevant role that these functions play in Fractional Calculus, namely in solving fractional differential equations that arise in the modeling of many practical problems in science and engineering.

In this paper, we deal with the two-parameter ML function (we omit the terminology ``two-parameter'' to simplify) and assume that $\alpha$ and $\beta$ are real and positive. This is the case that is relevant in most of the practical applications. Under those assumptions regarding $\alpha$ and $\beta$, it is well-known that $E_{\alpha,\beta}(z)$, with $z\in\mathbb{C}$, is an entire function which has two important representations suitable to be exploited for numerical computations purposes:
the Taylor series 
\begin{equation}\label{def-taylor}
	E_{\alpha,\beta}({z})=\sum_{k=0}^{\infty}\frac{{z}^{k}}{\Gamma(\alpha k+\beta)},
\end{equation}
and the Wiman's integral representation \cite{Wiman} and \cite[p. 210]{Erdelyi}:
\begin{equation}\label{def-wiman}
	E_{\alpha,\beta}({z})=\frac{1}{2\pi i}\int_{\mathcal C} \frac{y^{\alpha-\beta}e^y}{y^\alpha-z}\ dy,
\end{equation} 
where the contour ${\mathcal C}$ is the so-called Hankel path in the complex plane, that is, it is a path starting and
ending at $-\infty$ and encircling the disk $|y|\leq z^{1/\alpha}$ in the positive sense. 

The methods proposed in \cite{Gorenflo02} and \cite{Seybold08} are similar (though the method in \cite{Seybold08} makes significant improvements on the one in \cite{Gorenflo02}) and both use the truncation of the Taylor series \eqref{def-taylor} for approximating $E_{\alpha,\beta}({z})$, when $|z|$ is small. When $|z|$ is large, they use asymptotic formulae derived from \eqref{def-wiman}; for intermediate values of $|z|$, they use  the integral representation \eqref{def-wiman} directly. Both methods assume that $0< \alpha \leq 1$; for $\alpha>1$ they use a recursion formula that reduces the latter case to the former.

The method in \cite{Garrappa15} is different from the ones just mentioned, in the sense that it does not require different techniques depending on the magnitude of $z$ or $\alpha$. It involves the auxiliary function
$$e_{\alpha,\beta}(t;\lambda):=t^{\beta-1}\,E_{\alpha,\beta}(t^\alpha\,\lambda),\quad t>0, \ \lambda\in\mathbb{C},$$
and its integral representation in terms of the inverse Laplace transform. Through a combination of the Residue Theorem together with the choice of an optimal parabolic contour, that minimizes the error and the computational cost, Garrappa (2015) designed what is now considered the state-of-the-art method for evaluating $E_{\alpha,\beta}({z})$. MATLAB files for the methods in \cite{Gorenflo02} and \cite{Garrappa15} are available, respectively, in \cite{Podlubny2} and \cite{Garrappa-m}. 

To understand the influence of the values of $\alpha$ and $\beta$ on the magnitude of $E_{\alpha,\beta}({z})$, let us observe Figures \ref{fig1} and \ref{fig2}, where we see that, for $\alpha = 0.5$, the magnitude of the ML function blows up when $z$ is in the real line somewhere between $1.5$ and $2.5$. In contrast, if $\alpha$ increases slightly to $0.8$ and $\beta$ to $1.5$, then $|E_{\alpha,\beta}({z})|$ increases, although not as quickly. Analytically, this can be seen by looking at the denominator of the fractions in the sum \eqref{def-taylor}. For small values of $\alpha$, the value of the Gamma function $\Gamma(\alpha k+\beta)$ 
may not compensate for the growth of the powers of $z$ in the numerator. In both pictures, we see that the magnitude of the ML function is small whenever the real part of $z$ is negative. These kind of issues must be taken into account in the design of algorithms to the ML function. 

 
\begin{center}
\begin{figure}[t]
\hspace*{-1.5cm}\includegraphics[width=17cm]{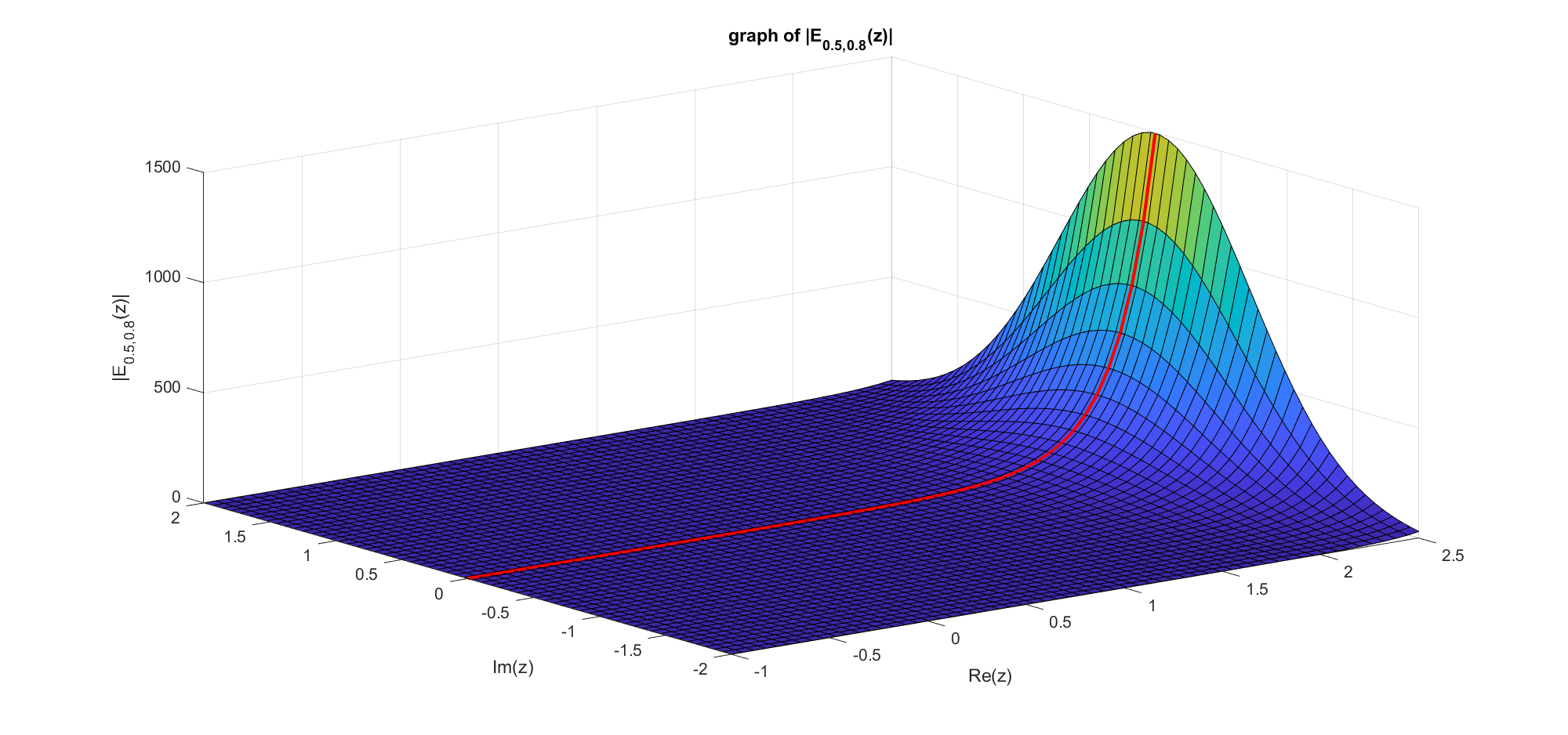}
\caption{\small Values of $|E_{\alpha,\beta}({z})|$, with $\alpha=0.5$ and $\beta=0.8$, for the  complex numbers $z$ belonging to the rectangle $[-1,2.5]\times [-2,2]$. The red line over the graph surface concerns to the values of the ML function on the real line. 
 }
\label{fig1}
\end{figure}
\end{center}

\begin{center}
\begin{figure}[t]
\centering
\includegraphics[width=14cm]{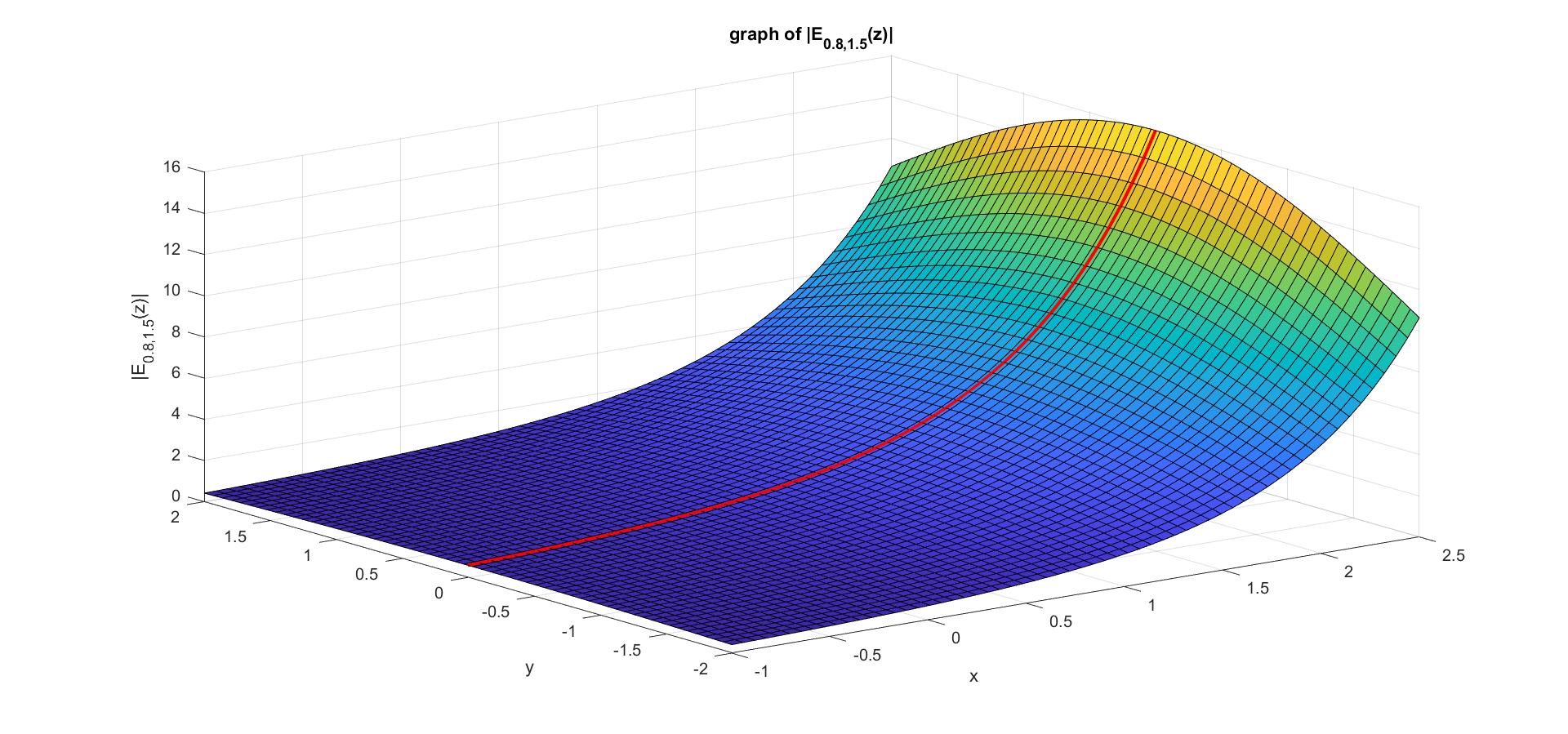}
\caption{\small Values of $|E_{\alpha,\beta}({z})|$, with $\alpha=0.8$ and $\beta=1.5$, for the  complex numbers $z$ belonging to the rectangle $[-1,2.5]\times [-2,2]$. The red line over the graph surface concerns to the values of the ML function on the real line. }
\label{fig2}
\end{figure}
\end{center}

\section{When is the Taylor Series Suitable For Numerical Computations?}\label{sec-taylor}

In this section, we state a couple of results regarding the practical usage of the Taylor series \eqref{ml-function} in numerical computations. Proposition \ref{prop-taylor-1} is an extension of \cite[Thm.4.1]{Seybold08} to matrices. Propositions \ref{prop-taylor-2} and \ref{prop-taylor-3} are, to the best of our knowledge, novel in the context of matrices as well as scalars. We recall that algorithms for computing matrix functions based on matrix polynomials are very interesting from a computational point of view in the sense that they are rich in matrix-matrix computations. Moreover, they can easily be extended to algorithms for  computing a matrix function times a vector in only $O(n^2)$ operations, avoiding the entire determination of $f(A)$. This is of great relevance if $A$ is large. 

Let $\|.\|$ be a given subordinate matrix norm and let us consider the truncation error of the Taylor series:
\begin{equation}\label{err-taylor}
	R_m(A) := \left\|E_{\alpha,\beta}(A)-\sum_{k=0}^{m}\frac{A^{k}}{\Gamma(\alpha k+\beta)}\right\|=\left\|\sum_{k=m+1}^{\infty}\frac{A^{k}}{\Gamma(\alpha k+\beta)}\right\|,
\end{equation}
where $A$ is a given complex square matrix and $\alpha$ and $\beta$ are real positive.

\begin{proposition}\label{prop-taylor-1}
Let $\epsilon>0$ and assume that $\|A\|<1$. If
$$m\geq \max \left\{\left\lceil \frac{2-\beta}{\alpha}+1\right\rceil,\quad  \left\lceil \frac{\ln(\epsilon(1-\|A\|))}{\ln(\|A\|)}\right\rceil+1\right\},$$
then $R_m(A)\leq \epsilon$, where $R_m(A)$ is given in \eqref{err-taylor} and $\lceil.\rceil$ denotes the ceiling of a number. 
\end{proposition}

\begin{proof}
	Follow the steps of the proof of Theorem 4.1 in \cite{Seybold08} and attend to the properties of matrix norms. \qed
\end{proof}
	
\begin{proposition}\label{prop-taylor-2}
	Let $\|A\|<1$ and let $\mu>0$. If there exists an integer $m$ such that 
	\begin{equation}\label{find-m}
		\frac{\|A^{m+1}\|}{\Gamma(\alpha (m+1)+\beta)}\leq \mu,
	\end{equation}
	then 
	$$R_m(A)\leq \mu \frac{\|A\|}{1-\|A\|}.$$ 
\end{proposition}

\begin{proof}
First, we note that $\|A^{m+2}\|\leq \|A\|\|A^{m+1}\|$ and $\Gamma(\alpha (m+1)+\beta)<\Gamma(\alpha (m+2)+\beta)$. Since, by assumption, $\|A^{m+1}\| \leq \mu\,\Gamma(\alpha (m+1)+\beta)$, the following is valid:
$$\frac{\|A^{m+2}\|}{\Gamma(\alpha (m+2)+\beta)}\leq \frac{\|A\|\|A^{m+1}\|}{\Gamma(\alpha (m+2)+\beta)}\leq \|A\|\frac{\|A^{m+1}\|}{\Gamma(\alpha (m+1)+\beta)}\leq \|A\| \mu.$$
By induction, it is easy to show that, for any integer $k>1$,
$$\frac{\|A^{m+k}\|}{\Gamma(\alpha (m+k)+\beta)}\leq \|A\|^{k-1} \mu.$$
Now, attending that $\|A\|<1$, the result follows from
\begin{eqnarray*}
R_m(A) & =& \left\|\sum_{k=m+1}^{\infty}\frac{A^{k}}{\Gamma(\alpha k+\beta)}\right\|\\
&\leq &\sum_{k=m+1}^{\infty}\frac{\left\|A^{k}\right\|}{\Gamma(\alpha k+\beta)}\\
& \leq& \mu \left(\|A\|+\|A\|^2+\cdots\right)\\
&=& \mu \frac{\|A\|}{1-\|A\|}.
\end{eqnarray*}
\qed
\end{proof}

One important point of the previous result is the possibility of estimating the absolute error \eqref{err-taylor} using the norm of the terms of the Taylor series, so that the smallest $m$ satisfying \eqref{find-m} gives the number of terms to be considered in the truncation of the Taylor series. 

The next result shows that the Taylor series \eqref{ml-function} may be used for numerical computations even when $\|A\|\geq 1$.

\begin{proposition}\label{prop-taylor-3}
	Let $m$ be a natural number and let $a\in\mathbb{R}^+$ such that
	 $\Gamma(\alpha k+\beta) \geq a^k$, for all integer $k\geq m$. If $b:=\|A\|/a<1$, then 
	 \begin{equation}\label{find-m-1}
	 R_m(A)\leq \frac{b^{m+1}}{1-b}.
	 \end{equation}
	 
\end{proposition}

\begin{proof}
Since, by assumption, there exists a positive integer number $m$ such that $\Gamma(\alpha k+\beta) \geq a^k$, for all integer $k\geq m$, and $b=\|A\|/a<1$, we have 
\begin{eqnarray*}
		R_m(A) &=&  \left\|E_{\alpha,\beta}(A)-\sum_{k=0}^{m}\frac{A^{k}}{\Gamma(\alpha k+\beta)}\right\|\\		
		& =& \left\|\sum_{k=m+1}^{\infty}\frac{A^{k}}{\Gamma(\alpha k+\beta)}\right\|\\
		&\leq &\sum_{k=m+1}^{\infty}\frac{\left\|A^{k}\right\|}{\Gamma(\alpha k+\beta)}\\
		& \leq& \sum_{k=m+1}^{\infty}\left(\frac{\|A\|}{a}\right)^k\\
		&=& \sum_{k=m+1}^{\infty} b^k\\
		&=& \frac{b^{m+1}}{1-b}.
	\end{eqnarray*}
	\qed 
\end{proof}

Now, we shall discuss how to use the series \eqref{ml-function} in the computation of $E_{\alpha,\beta}(A)$. An important point is that the norm of $\|A\|$ must be small enough; otherwise, $\|A^{k+1}\|$ may overflow, even for a small integer $k$, and catastrophic cancellation may occur when adding up the terms of the series. Another point is the possibility of the values of the Gamma function in the denominators of \eqref{ml-function} overflow too. It is known that in IEEE-754 double precision arithmetic the Gamma function
can be evaluated only for arguments smaller than or equal to $171.624$, so the maximum number of terms we can use in \eqref{ml-function} is given by (see \cite{Garrappa18})
$$m_{\max} := \left\lfloor\frac{171.624-\beta}{\alpha}\right\rfloor,$$
where $\lfloor.\rfloor$ denotes the floor of a number. 

For a practical implementation of Propositions \ref{prop-taylor-1}, \ref{prop-taylor-2} and \ref{prop-taylor-3}, we find more interest in the result of Proposition \ref{prop-taylor-3}, because it does not require the condition $\|A\|<1$.

In Algorithm \ref{algorithm1} we describe how to implement the result of Proposition \ref{prop-taylor-3} in a practical way. The main issue is in how to find an appropriate $a$ such that 
$\Gamma(\alpha m+\beta) \geq a^m$ holds for a small enough $m$. It is not convenient to use a large number of terms in the series. Since $ab=\|A\|\leq a$,
the value of $a$ is bounded below by $\|A\|$. However, it is not recommended to take  $\displaystyle a\gtrapprox\|A\|$ because $\displaystyle b\lessapprox 1$, which would lead to a poor estimate for the error in \eqref{find-m-1}. Another issue that must be considered is that the norm of $A$ needs to be bounded, according to the values of $\alpha$ and $\beta$, to guarantee convergence. To find such a bound, we will require that the norm of the term in \eqref{ml-function} corresponding to $m_{\max}$ is small enough, that is,
\begin{equation}\label{small-eta}
\frac{\|A\|^{m_{\max}}}{\Gamma(\alpha\,m_{\max}+\beta)}<\eta,
\end{equation}
where $\eta$ is a small number, say, $\eta=2^{-53}$. 
From \eqref{small-eta}, it follows
\begin{equation}\label{bound-norm}
	\|A\|<\left[\eta\Gamma\left(\alpha\,m_{\max}+\beta\right)\right]^{1/m_{\max}}.
\end{equation}
Hence, if $\|A\|$ does not meet the condition \eqref{bound-norm}, it is not guaranteed that \eqref{ml-function} will be appropriate for computing the ML function.

In Algorithm \ref{algorithm1}, we choose $a=2\|A\|$, which, according to a set of numerical experiments we have carried out, gives a good compromise between a small value for $m$ and a good error estimate in \eqref{find-m-1}. If the conditions in Step 9 are not satisfied, our recommendation is that Taylor series may give inaccurate results. This might happen in particular when $\alpha$ is small (say, $0<\alpha<1$) or $\|A\|$ is large. 

{\rm \begin{algorithm}\label{algorithm1}
 (Taylor series) Given $\alpha$ and $\beta$ positive real numbers, $A\in\mathbb{C}^{n\times n}$ and a tolerance $\epsilon$, this algorithm decides if the Taylor series \eqref{ml-function} is suitable for approximating $E_{\alpha,\beta}(A)$ and, if so, finds the number of terms that should be considered in the series. 
 
 \begin{enumerate}
\item $m_{\max} = \left\lfloor\frac{171.624-\beta}{\alpha}\right\rfloor$;
\item $norm\_max = (\epsilon\Gamma(\alpha \,m_{\max}+\beta))^{1/m_{\max}}$;
\item $ a = 2\|A\|$;
\item $ b = 1/2$;
\item $m = 1:m_{\max}$;
\item $ d = \Gamma(\alpha m+\beta)-a.\,\,\hat{}\,\, m$;
\item $ k_1 = find(d>0,1)$;
\item $ k_2 = \lceil \log(\epsilon(1-b))/\log(b)-1\rceil$ (see \eqref{find-m-1} );
\item  {\bf if} $\|A\|\leq norm\_max$ and $k_1$ is non-empty and $k_1\leq k_2$
\item $\quad$ Use Taylor series \eqref{ml-function} with $k_2$ terms for approximating $E_{\alpha,\beta}(A)$; 
\item {\bf else} Taylor series is not recommended for evaluating $E_{\alpha,\beta}(A)$.
\item {\bf end}
\end{enumerate}
\end{algorithm}
}
  
Provided that the conditions in Step 9 of Algorithm \ref{algorithm1} are satisfied, there is the question of how to evaluate the Taylor polynomial of degree $k_2$. There are three main methods for polynomial evaluation: Horner's method, the explicit powers' method and the Paterson-Stockmeyer method. The first two methods cost about $k_2-1$ matrix multiplications, while the latter is more sophisticated and much less expensive, costing about $s+r-1-\phi(s,k_2)$ matrix multiplications, where $s=\lceil\sqrt{k_2}\rceil$ or $s=\lfloor \sqrt{k_2}\rfloor$, $r=\lfloor \sqrt{k_2/s}\rfloor$ and $\phi(s,k_2)=1$ if $s$ divides $k_2$ and $0$ otherwise. In the method to be proposed later in Section \ref{sec-experiments} we will use the Paterson-Stockmeyer method.
For more information on evaluating matrix polynomials, see \cite[Ch.4]{Higham08}; for Paterson-Stockmeyer method see also \cite[Ch.4]{Higham08}, and \cite{Fasi,Paterson}.

\section{Finding a Suitable Contour}\label{contour}

In this section we assume throughout that $T$ is an atomic block (i.e., an upper triangular matrix with close eigenvalues) of order $n$. Our main goal is to provide a reliable method for computing the ML function of $T$ by means of the Cauchy formula
$$
E_{\alpha,\beta}(T)=\frac{1}{2\pi i}\int_{\mathcal C} E_{\alpha,\beta}(z)(zI-T)^{-1}\ dz.
$$
For the contour ${\mathcal C}$ we consider a circle with center in $z_0\in\mathbb{C}$ and radius $r>0$,
\begin{equation}\label{circle}
	{\mathcal C}=\{z\in\mathbb{C}:\ z=z_0+re^{ti}, \ t\in [0,2\pi[\},
\end{equation}
enclosing the spectrum $\sigma(T)$ of $T$. After a few calculations we obtain 
\begin{equation}\label{cauchy-2}
E_{\alpha,\beta}(T)=\frac{1}{2\pi}\int_0^{2\pi} re^{ti}E_{\alpha,\beta}(z_0+re^{ti})
\left((z_0+re^{ti})I-T\right)^{-1}\ dt.
\end{equation}
For the center of the circle we set 
\begin{equation}\label{z0}
	z_0 = \frac{\trace(T)}{n},
\end{equation}
that is, $z_0$ is the arithmetic mean of the eigenvalues of $T$. The value of the radius $r$ must satisfy
\begin{equation}\label{d}
	r\,>\,d:=\max\{|\lambda-z_0|:\ \lambda\in \sigma(T)\};
\end{equation}
see Figure \ref{fig-4}.

\begin{center}
	\begin{figure}[t]
		\centering
		\includegraphics[width=0.30\textwidth]{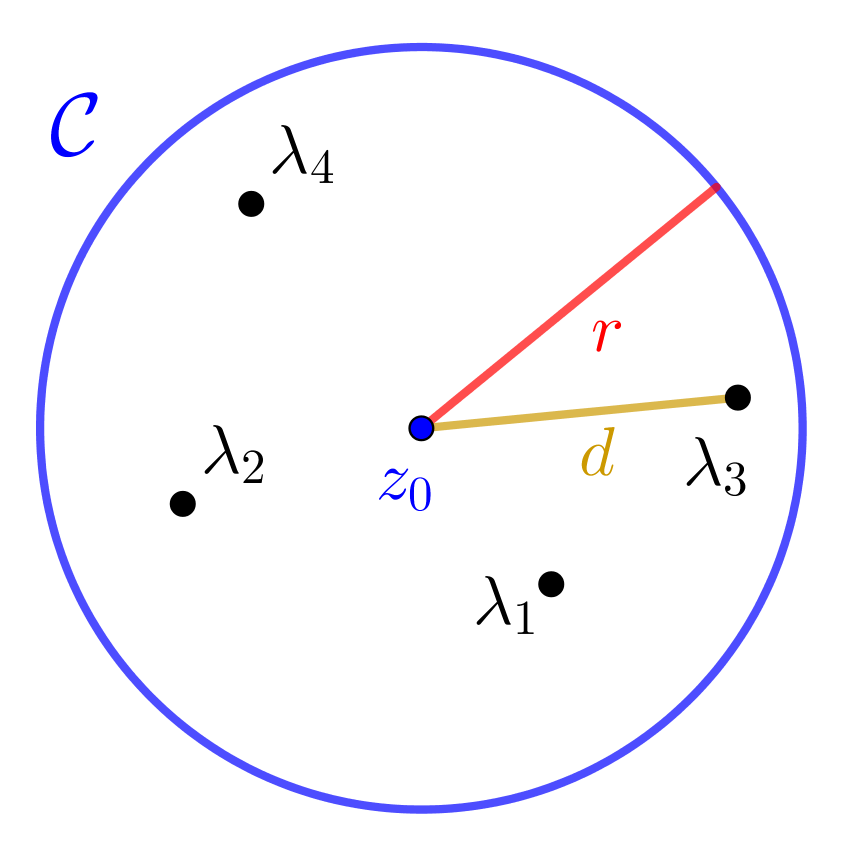}
		\caption{\small Geometric representation of $d$ and $r$ in \eqref{d}, for a matrix $T$ with spectrum $\sigma(T) = \{\lambda_1,\, \lambda_2,\, \lambda_3,\, \lambda_4\}$. }
		\label{fig-4}
	\end{figure}
\end{center}

Since we are interested in evaluating the integral in \eqref{cauchy-2} by quadrature, some care must be taken in the choice of the radius $r$. If it is very close to $d$, the circle passes too close the eigenvalues of $T$ and the value of the norm of the resolvent $\left\|\left((z_0+re^{ti})I-T\right)^{-1}\right\|$ is large, while the value of $E_{\alpha,\beta}(z_0+re^{ti})$ may be moderate; if $r$ is large so that the circle stays far away from the eigenvalues of $T$, the value of $\left\|\left((z_0+re^{ti})I-T\right)^{-1}\right\|$ is moderate, while the values of $E_{\alpha,\beta}(z_0+re^{ti})$ may be large, especially if $0<\alpha <1$. Therefore the value of $r$ must be carefully chosen in order to avoid many oscillations in the integrand function of \eqref{cauchy-2}. Moreover, it is expected that $r$ should minimize the quadrature error arising in the approximation of \eqref{cauchy-2}. 

Let 
\begin{equation}\label{integrand}
g(t):=re^{ti}E_{\alpha,\beta}(z_0+re^{ti})\left((z_0+re^{ti})I-T\right)^{-1}
\end{equation}
be the integrand function in \eqref{cauchy-2}. Note that $g:\mathbb{R}\longrightarrow \mathbb{C}^{n\times n}$, where $n$ is the order of $T$, is of class $C^\infty$ (we are assuming that the eigenvalues of $T$ are located in the interior of the circle). Since $g$ is periodic, we follow the recommendations of \cite{Trefethen} and use the composite trapezoidal rule (or just trapezoidal rule, for short).
Let 
\begin{equation}\label{trapez-1}
{\mathcal T}_m(g):=\frac{2\pi}{m}\sum_{k=1}^{m}g(t_k),
\end{equation}
denote the $(m+1)$-point trapezoidal rule applied to $\int_0^{2\pi}g(t)\ dt$. Here, $t_0,t_1,\ldots,t_m$ are equally spaced points partitioning the interval $[0,\,2\pi]$, with $t_0=0$ and $t_m=2\pi$. The point $t_0$ does not appear explicitly in \eqref{trapez-1} because $g(t_0)=g(t_m)$. Let 
\begin{equation}\label{eps-T}
	\epsilon_T:=\left\|\int_0^{2\pi}g(t)\ dt-{\mathcal T}_m(g)\right\|
\end{equation}
denote the error of the $(m+1)$-point trapezoidal rule. In the literature we can find two important bounds for such an error: 
\begin{equation}\label{err-1}
	\epsilon_T \leq \frac{8\pi^3}{12m^2}\,M_1
\end{equation}
(see \cite[Sec. 4]{Cardoso}),
and 
\begin{equation}\label{err-2}
\epsilon_T \leq \frac{4\pi \,\zeta(2s+1)}{m^{2s+1}}\,M_2
\end{equation}
where $\displaystyle M_1:=\max_{t\in [0,\,2\pi]}\|g''(t)\|$, $\displaystyle M_2:= \max_{t\in [0,2\pi]}\|g^{(2s+1)}(t)\|$  and $\zeta(.)$ is the Riemann Zeta function (see \cite{Davies05,Davis84}). The bound in \eqref{err-2} is valid because $g$ is periodic, while \eqref{err-1} does not require the periodicity of $g$. Of course, the bound \eqref{err-2} is in general sharper than \eqref{err-1}, but the expression of the derivatives $g^{(k)}(t)$ becomes much complicated as $k$ increases. Using Leibniz's rule and Fa\`a di Bruno's formula, it is possible to find a closed expression for $g^{(k)}(t)$, but its computational cost for each $t$ increases with $k$ as well.

To illustrate, we write down the expression of $g''(t)$. Attending that $g$ is the product of three functions, that is, $g(t)=g_1(t)g_2(t)g_3(t)$, where
\begin{eqnarray}
	g_1(t)&:=&re^{ti}\label{g1}\\
	g_2(t)&:=&E_{\alpha,\beta}(z_0+re^{ti})\label{g2}\\
	g_3(t)&:=&\left((z_0+re^{ti})I-T\right)^{-1},\label{g3}
\end{eqnarray}
the second order derivative is given by
\begin{equation}\label{der-g-2}
g''(t)=\left[g_1(t)g''_2(t)+2g'_1(t)g'_2(t)+g_1(t)g''_2(t)\right]g_3(t)+
2\left[g'_1(t)g_2(t)+g_1(t)g'_2(t)\right]g'_3(t)+g_1(t)g_2(t)g''_3(t),
\end{equation}
where
\begin{eqnarray*}
	g'_1(t)&=& rie^{ti}\\
	g''_1(t)&=& ri^2e^{ti}\\
	g'_2(t)&=& \frac{dE_{\alpha,\beta}(z)}{dz}(rie^{ti})\\
	g''_2(t)&=& \frac{d^2E_{\alpha,\beta}(z)}{dz^2}(rie^{ti})^2+ \frac{dE_{\alpha,\beta}(z)}{dz}(ri^2e^{ti})\\
	g'_3(t)&=&-rie^{ti}(zI-T)^{-2}\\
	g''_3(t)&=&re^{ti}(zI-T)^{-2}+2\left(rie^{ti}\right)^2(zI-T)^{-3},
\end{eqnarray*}
with $z=z_0+re^{ti}$. If we go up to the third order derivative, we arrive at a much more complicated expression, so in the following developments we will use \eqref{err-1} instead of \eqref{err-2} and focus our investigation on how to find an $r$ that minimizes \eqref{err-1}.

Mathematically, the problem can be formulated as a min-max problem 
\begin{equation}\label{min-max}
	\tilde{r}=\argmin_{r>d}\ \max_{t\in [0,2\pi]} \|g''(r,t)\|,
\end{equation}
where $d$ is defined in \eqref{d}, that is, $d$ is the maximal distance between $z_0$ and the eigenvalues of $T$ (see \eqref{z0} and also Figure \ref{fig-4}).
For convenience, from now on we highlight the dependence of $g''$ not only on $t$, but also on $r$. Since a rough estimate for $\tilde{r}$ is enough, we can solve \eqref{min-max} by considering a vector of values $r\in ]d,3]$ and then finding the maximal $t$ corresponding to each $r$. At the end, we choose an $r$ corresponding to the  minimum of the maximal values of $t$; this will be the radius of the contour circle.

%
%

Solving the optimization problem \eqref{min-max} requires many evaluations of  $g''(r,t)$, at the cost of $O(n^3)$ operations per evaluation. Note that $g''(r,t)$  involves the inversion of a triangular matrix in $\left((z_0+re^{ti})I-T\right)^{-k}$, for $k=1,2$, at the cost of $O(n^3)$ operations, while the functions $g_1$ and $g_2$ (and their derivatives) are scalar valued. 

To make each function evaluation cheaper and more feasible, in the following proposition we derive an upper bound for $\left\|\left((z_0+re^{ti})I-T\right)^{-k}\right\|_F$, where $\|.\|_F$ denotes the Frobenius norm.

\begin{proposition}
Given an $n\times n$ complex upper triangular matrix $T$, let $N$ be its strictly upper triangular part, that is, $T=\diag(\lambda_i)+N$. Let ${\mathcal C}$ be the circle with center in $z_0=\trace(T)/n$ and radius $r$ and let $\widetilde{\mathcal C}$ be the closed disk with center in $z_0$ and radius $\delta$,  such that $r>\delta>d=\max\{|\lambda-z_0|:\ \lambda\in \sigma(T)\}$. Then, for all $t\in [0,\,2\pi]$ and $k=1,2,\ldots$,
\begin{equation}\label{bound-g3}
	\left\|\left((z_0+re^{ti})I-T\right)^{-k}\right\|_F 
	\leq c\, \|(I-|N|)^{-1}\|_F,
\end{equation}
where 
$$c:=\max\left\{\frac{(k+j-1)!}{j!(k-1)!}(r-\delta)^{-k-j}:\ j=0,1,\ldots,n-1\right\}$$ and $|N|$ denotes the matrix formed by the absolute values of the entries of $N$.
\end{proposition}

\begin{proof}
	Let us consider the scalar complex valued function $\phi_k(x):=\frac{1}{(z(t)-x)^k}$, where $k=1,2,\ldots$ and $z(t):= z_0+re^{ti}$. It is easy to check that $\phi_k(x)$ is analytic on the closed disk $\widetilde{\mathcal C}$. Note that if $\lambda$ is an eigenvalue of $T$, then $z(t)-\lambda\neq 0, \ \forall t\in[0,\,2\pi]$, that is, the denominator of $\phi_k(\lambda)$ does not vanish. Using the result \cite[Thm.9.2.2]{Golub} (see also \cite[Thm.4.28]{Higham08}), 
	we have
	\begin{equation}\label{ineq-1}
		\left\|\left((z_0+re^{ti})I-T\right)^{-k}\right\|_F
		\leq  \max_{0\leq j\leq n-1}\frac{\omega_j}{j!}\,\|(I-|N|)^{-1}\|_F,
		\end{equation}
	where $ \displaystyle\omega_j:=\sup_{x\in \widetilde{\mathcal C}}|\phi_k^{(j)}(x)|.$
		Noticing that 
		$\phi_k^{(j)}(x)=\frac{(-1)^j (k+j-1)!}{(k-1)!\,(z-x)^{k+j}}$ and that $|z(t)-x|\geq r-\delta$, for any $x\in \widetilde{\mathcal C}$ and $t\in [0,\,2\pi]$, we have 
		$$ \frac{\omega_j}{j!} =\sup_{x\in \widetilde{\mathcal C}} \frac{(k+j-1)!}{j!(k-1)!\,|z-x|^{k+j}} \leq \frac{(k+j-1)!}{j!(k-1)!\,(r-\delta)^{k+j}},$$
		for each $j$ and $k$.
			Hence
			$$\left\|\left(z(t)I-T\right)^{-k}\right\|_F 
		  \leq \max\left\{\frac{(k+j-1)!}{j!(k-1)!}(r-\delta)^{-k-j}:\ j=0,1,\ldots,n-1\right\}\,\|(I-|N|)^{-1}\|_F,$$
		 which proves the result. \qed
\end{proof}
	
From \eqref{der-g-2} and \eqref{bound-g3}, we have
\begin{eqnarray}
	\|g''(r,t)\|_F &\leq&  \Big(\left|g_1(t)g''_2(t)+2g'_1(t)g'_2(t)+g_1(t)g''_2(t)\right|\gamma_1+ 2r\left|g'_1(t)g_2(t)+g_1(t)g'_2(t)\right|\gamma_2+\nonumber\\
	&&	r\left|g_1(t)g_2(t)\right|\left(\gamma_2+2r\gamma_3\right)\Big)\gamma,\label{bound-der-2}
\end{eqnarray}
where
\begin{eqnarray*}
	\gamma &:=& \|(I-|N|)^{-1}\|_F; \\
	\gamma_1 &:=& \max\left\{(r-\delta)^{-j}:\ j=1,2,\ldots,n\right\};\\
	\gamma_2 &:=& \max\left\{j(r-\delta)^{-j-1}:\ j=1,2,\ldots,n\right\}; \\
	\gamma_3 &:=& \max\left\{\frac{j(j+1)}{2}(r-\delta)^{-j-2}:\ j=1,2,\ldots,n\right\}.	
\end{eqnarray*}

Using the bound \eqref{bound-der-2} for estimating $\|g''(r,t)\|_F$ avoids the computation of the factor \linebreak $\left\|\left((z_0+re^{ti})I-T\right)^{-1}\right\|_F$, that requires $O(n^3)$ operations, for each $t$, making the estimation of  $\|g''(r,t)\|_F$ much cheaper: the factor $\|(I-|N|)^{-1}\|_F$ does not depend on either $r$ or $t$ and can be discarded. Hence, the operations needed to estimate a solution of \eqref{min-max} just involve scalars. 


\section{Algorithm}\label{sec-algorithm}

Before stating the algorithm for computing the ML function of a general matrix $A$ (Algorithm \ref{algorithm3}), we shall revisit the blocked Schur-Parlett method with reordering and explain in detail how to efficiently implement  the trapezoidal rule.

\subsection{Schur-Parlett Approach}\label{schur-parlett}
In the following, we summarize the main steps of the blocked Schur-Parlett method with reordering proposed in \cite{Davies03}, which codes are available in \cite{mftoolbox}; see also \cite[Ch. 9]{Higham08}. It starts by computing a Schur decomposition 
$A=UTU^\ast$, where $U$ is unitary and 
\begin{equation} \label{matrizT}
	T=\left[
	\begin{array}{cccc}
		T_{11}&T_{12}&\cdots&T_{1q}\\
		0&T_{22}&\cdots&T_{2q}\\\vdots&\ddots&\ddots&\vdots\\0&\cdots&0&T_{qq}
	\end{array}
	\right]\ \in\mathbb{C}^{n\times n},
\end{equation} is a $(q\times q)$-block-upper triangular matrix. Each atomic block $T_{ii}\ (i=1,\cdots,q)$  has clustered eigenvalues and the $q$ atomic blocks in the diagonal do not have common eigenvalues, that is, 
\begin{equation}\label{espectroT}
\sigma(T_{ii})\cap\sigma(T_{jj})=\emptyset,\ i,j=1,\ldots,q,\ i\neq j.
\end{equation}
Let us denote 
\begin{equation}\label{L1}
	F:=E_{\alpha,\beta}(T)=\left[
	\begin{array}{cccc}
		F_{11}&F_{12}&\cdots&F_{1q}\\
		0&F_{22}&\cdots&F_{2q}\\\vdots&\ddots&\ddots&\vdots\\0&\cdots&0&F_{qq}
	\end{array}
	\right], \end{equation} where $F_{ij}$ has the same size as $T_{ij}\ (i,j=\,\ldots,q)$. Recall that the diagonal blocks of $F$ are given by
$F_{ii}=E_{\alpha,\beta}(T_{ii})$. Since $FT=TF$, it can be shown that 
\begin{equation}\label{parlett2}
	F_{ij}T_{jj}-T_{ii}F_{ij}=T_{ij}F_{jj}-F_{ii}T_{ij}+
	\sum^{j-1}_{k=i+1}(T_{ik}F_{kj}-F_{ik}T_{kj})\quad i<j.
\end{equation}
To find the blocks of $F$, we start by computing the ML function of the atomic blocks $F_{ii}=E_{\alpha,\beta}(T_{ii})$ and then successively use  (\ref{parlett2}) to approximate the remaining blocks of $F$. Note that for each $i<j$, the identity (\ref{parlett2}) is a Sylvester equation of the form 
\begin{equation} \label{sylvester}
	XM-NX=P,
\end{equation} where $M$, $N$ and $P$ are known square matrices and $X$ has to be determined. Equation 
(\ref{sylvester}) has a unique solution if and only if $\sigma(M)\cap\sigma(N)=\emptyset$. Hence, the block-Parlett method requires the solving of several Sylvester equations with a unique solution. Recall that we are assuming that $\sigma(T_{ii})\cap\sigma(T_{jj})=\emptyset,\ i\neq j$. To avoid the ill conditioning of the Sylvester equations arising in the Parlett recurrence, the eigenvalues of the blocks $T_{ii}$ and $T_{jj}$, $i\neq j,$ need to be well separated in the following sense:
there exists $\delta>0$ (e.g., $\delta=0.1$), such that 
\begin{equation}\label{delta}
	\min\left\{|\lambda-\mu|:\ \lambda\in\sigma(T_{ii}),\ \mu\in\sigma(T_{jj}),\ i\neq j \right\} > \delta
\end{equation}
and, for every eigenvalue $\lambda$ of a block $T_{ii}$ with dimension bigger than $1$, there exists $\mu\in\sigma(T_{ii})$ such that $|\lambda-\mu|\leq \delta$. 

\subsection{Trapezoidal Rule} 

 The benefits of using the trapezoidal rule in integrals with periodic functions are clearly pointed out in \cite{Trefethen}. We highlight in particular the following three major points:
\begin{itemize}
	\item The computations of ${\mathcal T}_m(g)$ (see \eqref{trapez-1}) can be reused when the number of nodes is doubled, i. e., for computing ${\mathcal T}_{2m}(g)$;	
	\item  The error $\epsilon_T$ in \eqref{eps-T} can be estimated cheaply by
	$$\epsilon_T\leq \left\|{\mathcal T}_{2m}(g)-{\mathcal T}_{m}(g)\right\|;$$
	see \cite[Sec.2.3]{Tatsuoka} and also \cite[Sec.4]{Cardoso};
	\item If $T$ is a real matrix, the computation of ${\mathcal T}_m(g)$ can be carried out with only $\lceil m/2\rceil$ function evaluations.
\end{itemize}

For implementing the trapezoidal rule,  we consider the following algorithm.

{\rm \begin{algorithm}\label{algorithm2}
		(Trapezoidal rule) Assume that $T\in\mathbb{C}^{n\times n}$ is an upper triangular matrix and that $g(t)$ is defined as in \eqref{integrand}. 
		Let $\displaystyle{\mathcal T}_{m}(g):=\frac{2\pi}{m}\sum_{k=1}^{m}g(t_k),$ where 
		$t_0=0<t_1<\ldots<t_m=2\pi$ are equally spaced points, and let $\displaystyle{\mathcal I}_g:=\int_0^{2\pi}\,g(t)\,dt$ be the exact value of the integral.
		For a certain tolerance $\varepsilon$, this algorithm approximates the matrix integral by the trapezoidal rule and finds an $m_0$ such that the discretization error $\|{\mathcal I}_g-{\mathcal T}_{m_0}(g)\|_F \leq \varepsilon$.
		\begin{enumerate}
			\item Choose a positive integer $m$ (preferably small, say $m=10$);
			\item Compute ${\mathcal T}_{m}(g)$;
			\item Compute ${\mathcal T}_{2m}(g)$ by the following formula:
			 \begin{equation}\label{T-2m}
			{\mathcal T}_{2m}(g) = {\mathcal T}_{m}(g)+\sum_{k=0}^{m-1}g(t_{2k+1});
			\end{equation}
			\item {\bf while} $\|{\mathcal T}_{2m}(g)-{\mathcal T}_{m}(g)\|_F>\epsilon$,
			\item $\quad m\leftarrow 2m$;
			\item $\quad$Compute ${\mathcal T}_{2m}(g)$ using \eqref{T-2m};
			\item {\bf end}
			\item $m_0 \leftarrow m$;
			\item ${\mathcal I}_g\approx {\mathcal T}_{m_0}(g)$
		\end{enumerate}
	\end{algorithm}
}

To simplify, in Algorithm \ref{algorithm2} we have not exploited the advantages arising when $T$ is a real matrix. The main differences would be in the choice of $m$, that must be even, and in the way as ${\mathcal T}_{m}(g)$ is evaluated -- we would compute instead $$\widetilde{\mathcal T}_{m/2} :=\frac{\pi}{m}\left(g(0)+g(\pi)+2\sum_{k=1}^{m/2-1}g(t_k)\right),$$ where $t_0=0<t_1<\ldots<t_{m/2}=\pi$ are equally spaced points partitioning the interval $[0,\,\pi]$, and then evaluate
$${\mathcal T}_{m}=\widetilde{\mathcal T}_{m/2}+\conj\left(\widetilde{\mathcal T}_{m/2}\right);$$
here $\conj(.)$ denotes the complex conjugate.

\subsection{Main Algorithm}\label{alg-main}

{\rm \begin{algorithm}\label{algorithm3}
		(Computing the ML function) Given $\alpha$ and $\beta$ positive real numbers, and a general square complex matrix $A\in\mathbb{C}^{n\times n}$, this algorithm approximates $E_{\alpha,\beta}(A)$ using the Taylor series or a Schur-Parlett method combined with the Cauchy integral form. It is intended for working in IEEE standard double precision arithmetic.		
		\begin{enumerate}
			\item Call Algorithm \ref{algorithm1} to check if the Taylor series is suitable for computing the ML function and, if so, approximate $E_{\alpha,\beta}(A)$ by the Taylor polynomial of degree $k_2$, computed by means of the Paterson-Stockmeyer method; if the Taylor series cannot be applied, go to the next step;   
			\item Use the method of \cite{Davies03} to compute the blocked Schur decomposition of $A$ with reordering, so that  
			$A=UTU^\ast$, where $U$ is unitary and 
			$T$ is a $(q\times q)$-block-upper triangular matrix, with well separated atomic blocks $T_{kk},\ k=1,\ldots,q$ in its diagonal; choose $\delta=0.1$ for the value in \eqref{delta}. 
			\item {\bf for} $k=1:q$
			\item $\quad$ Let $n_k$ denote the order of the block $T_{kk}$;
			\item $\quad$ {\bf if} $n_k=1$, 
			 \item $\qquad$ $T_{kk}$ is a scalar so use the method of \cite{Garrappa15}; 
			\item $\quad$ {\bf endif}
			\item $\quad$ {\bf if} $n_k=2$, 
			\item $\qquad$ 
				let $T_{kk}:=[t_{11}\ t_{12};\,0\ t_{22}]$; now $F=E_{\alpha,\beta}(T_{kk})= [f_{11}\ f_{12};\,0\ f_{22}]$, where
			 	 $f_{ii}=E_{\alpha,\beta}(t_{ii})$, for $i=1,2$, 
			\item[] $\qquad$	and $f_{12}=t_{12}\frac{f_{22}-f_{11}}{t_{22}-t_{11}};$
			\item $\quad$ {\bf endif}
			\item $\quad$ {\bf if} $n_k\geq 3$
			\item $\qquad$ Set $z_k=\frac{\trace(T_{kk})}{n_k}$ and estimate the radius $r_k$ by the method of Section \ref{contour};
			\item $\qquad$ Denote 
			\begin{equation}\label{Gkk}
			G_k(t):=r_ke^{ti}E_{\alpha,\beta}(z_k+r_ke^{ti})\left((z_k+r_ke^{ti})I-T_{kk}\right)^{-1}; 
			\end{equation}
			 \item $\qquad$ Find $F_{kk}\approx \int_{0}^{2\pi}\,G_k(t)\,dt$ by the trapezoidal rule using Algorithm \ref{algorithm2};
			 \item $\quad$ {\bf endif} 
			  \item {\bf endfor}
			 \item Compute the remaining blocks of $F=2\pi\,E_{\alpha,\beta}(T)$, by the block Parlett recurrence (Section \ref{schur-parlett});
			 \item $E_{\alpha,\beta}(A)\approx \frac{1}{2\pi} U\,F\,U^\ast$.
			\end{enumerate}
	\end{algorithm}
}

Let us now make some remarks on the computational cost of Algorithm \ref{algorithm3}. If the Taylor series is used, a Taylor polynomial of degree $k_2$ is evaluated by the Paterson-Stockmeyer method whose cost has been discussed at the end of Section \ref{sec-taylor}. If the Taylor series cannot be used, a Schur decomposition is required which costs about $25n^3$ flops. Moreover, for the blocks $T_{kk}$ with size $n_k\times n_k$, where $n_k\geq 3$, the algorithm requires many evaluations of the function $G_k(t)$ given in \eqref{Gkk}, to compute the Cauchy integral by the trapezoidal rule. Each evaluation of $G_k(t)$ involves an inversion of a triangular matrix (or a multiple right-hand side linear system) which costs about $n_k^3/3$. Recall that $n_k$ is the order of the block $T_{kk}$ and that  $n_1+n_2+\cdots+n_q=n$, so in general $n_k$ is much smaller than $n$. The other factors of $G_k(t)$ are scalars. The estimation of the radius $r_k$ depends on the method used for solving the min-max problem \eqref{min-max}. Since we do not require a very accurate value for $r_k$, we can define a grid in the rectangle $[d,\tilde{d}]\times [0,\,2\pi]$ ($\tilde{d}$ is an upper bound for the admissible values for $r$; in our experiments we consider $\tilde{d}=3$), with ``large cells'' and then evaluate the bound for $\|G_k(r,t)\|_F$ given in \eqref{bound-der-2} at  each point $(r,t)$. Recall that the factor $\|(I-|N_{kk}|)^{-1}\|_F$ does not need to be computed. We consider a grid with ``large cells'' because the computation of the scalar ML function is quite expensive, if compared with functions like $\exp(z)$ or $\log(z)$.  
The cost of implementing the Schur-Parlett method is addressed in detail in \cite{Davies03} and \cite[Ch.9]{Higham08}. From these two latter references, we know that the Schur-Parlett method is forward stable, so the major source of errors in Algorithm \ref{algorithm3} arises in the computation of the Taylor polynomial or in the computation of the Cauchy integral by the trapezoidal rule, which is bounded by $\left\|{\mathcal T}_{2m}(g)-{\mathcal T}_{m}(g)\right\|$.   
If the Taylor series is suitable for evaluating the ML function, we can monitor the truncation error by the bound provided in Proposition \ref{prop-taylor-3}. 

\section{Numerical Experiments}\label{sec-experiments}

In this section we run a set of experiments in MATLAB R2022a on a machine with Core i5 (2.40 GHz, 16 GB of RAM, unit roundoff $u\approx 1.1\times 10^{-16}$) to compare the relative errors originated by Algorithm \ref{algorithm3} and the method of \cite{Garrappa18}, here named \texttt{GP-method}. The CPU time (in seconds) of both methods is compared as well. Since CPU time depends on many factors and is different each time an algorithm runs, we have run each algorithm 20 times and then computed the average CPU time to get more realistic results. 

The method in \cite{Higham21} is not considered in our experiments because it involves variable precision so we can control better the accuracy by changing the precision. However, improving the accuracy may be done at the cost of increasing the computational time and this issue is not addressed in \cite{Higham21}. Moreover, the ``Advanpix'' package required by the methods of \cite{Higham21} is not free and it is not incorporated in MATLAB's standard packages. Codes for the \texttt{GP-method} are available in \cite{Garrappa-m}. As exact value for $E_{\alpha,\beta}(A)$, we have considered the result given by the Symbolic Math Toolbox at 2000 decimal digit arithmetic rounded to double precision. Denoting by $\widetilde{E}_{\alpha,\beta}(A)$ the computed approximation, we use the relative error in the usual sense:
$$\mbox{\texttt{rel-err}}:=\frac{\|{E}_{\alpha,\beta}(A)-\widetilde{E}_{\alpha,\beta}(A)\|_F}{\|{E}_{\alpha,\beta}(A)\|_F}.$$
We recall that in \cite{Garrappa18} and also in \cite{Garrappa15} the authors consider a different formula for the relative error:
$$\mbox{\texttt{err-GP}}:=\frac{\|{E}_{\alpha,\beta}(A)-\widetilde{E}_{\alpha,\beta}(A)\|_F}{\|{E}_{\alpha,\beta}(A)\|_F+1}.$$
As discussed in \cite{Higham21}, if $\|E_{\alpha,\beta}(A)\|\gg 1$,  $\mbox{\texttt{err-GP}}$ gives a good estimate for the relative error; but if $\|E_{\alpha,\beta}(A)\|\ll 1$, $\mbox{\texttt{err-GP}}$ does not give a realistic estimate of the relative error; it estimates instead the absolute error.

To assess the forward stability of both algorithms, the relative errors are compared with the quantity $\cond_{ML}(A)u$, where $\cond_{ML}(A)$ denotes the relative condition number of the ML function at $A$ and $u$ is the unit  roundoff \cite[Ch.3]{Higham08}. The value of $\cond_{ML}(A)$ will be estimated by the code \texttt{funm\_condest1} available in \cite{mftoolbox}. Although the papers \cite{Garrappa15,Garrappa18} provide state-of-the-art methods for approximating the scalar ML function and the matrix ML function, respectively, we must be aware that they may not give relative errors close to $\cond_{ML}(A)u$ when $\|E_{\alpha,\beta}(A)\|\ll 1$. This might also happen when Algorithm \ref{algorithm3} calls the Cauchy integral form, because it requires many evaluations of the scalar ML function $E_{\alpha,\beta}(z)$, which will be carried out by means of the method provided in \cite{Garrappa15}.  

In Algorithm \ref{algorithm1}, we have chosen $\epsilon=10^{-15}$, so that $ k_2 = \lceil \log(\epsilon(1-b))/\log(b)-1\rceil = 50$. Hence, provided that the conditions of applicability of the Taylor series are satisfied,  $E_{\alpha,\beta}(A)$ is approximated by the Taylor polynomial of degree $k_2=50$, which is evaluated by the Paterson-Stockmeyer method, at the cost of about $13$ matrix multiplications. Since the cost of a Schur decomposition is about $25n^3$, corresponding to $12.5$ matrix multiplications, it is not convenient to use higher-degree polynomials to guarantee the competitivity of the Taylor method.   

\medskip\noindent {\it Experiment 1.} This experiment involves the $20\times 20$ matrix \texttt{A = gallery('redheff',20)} from MATLAB's gallery. This is the Redheffer matrix whose blocked Schur decomposition with reordering has a large block with eigenvalue equal to $1$. It has also been used in the experiments of \cite{Garrappa18,Higham21}. Figure \ref{fig-6} displays the relative errors arising in the computation of ${E}_{\alpha,\beta}(-A)$, with $\alpha$ having fixed values ($\alpha=0.5$ and $\alpha=0.8$) as $\beta$ varies from $1$ to $10$. As observed previously in \cite{Higham21}, the relative errors of the \texttt{GP-method} start increasing as $\beta$ goes from $6$ to $10$. This is in part because $\|{E}_{\alpha,\beta}(-A)\|_F$ approaches zero as $\beta$ grows toward $10$. As mentioned above, this might also occur when Algorithm \ref{algorithm3} uses the Cauchy integral formula. This is the case for the ten tests corresponding to  $\alpha=0.5$ (left graph), because in all of them Algorithm \ref{algorithm3} calls the Cauchy integral form combined with the Schur-Parlett method. We note that for
$\alpha=0.5$ and $\beta\in[6,10]$, one has $\|{E}_{\alpha,\beta}(-A)\|_F\in[10^{-1},\,10^{-5}]$. A quite different situation occurs for $\alpha=0.8$ (right graph), because now, for $\beta=5,\ldots,10$, Algorithm \ref{algorithm3} uses the Taylor series, thus preventing the loss of stability. Figure \ref{fig-6a} displays the corresponding CPU time taken by Algorithm \ref{algorithm3} and \texttt{GP-method}, where we can see that Algorithm \ref{algorithm3} has the best performance.

\begin{center}
	\begin{figure}[t]
		\centering
		\includegraphics[width=0.99\textwidth]{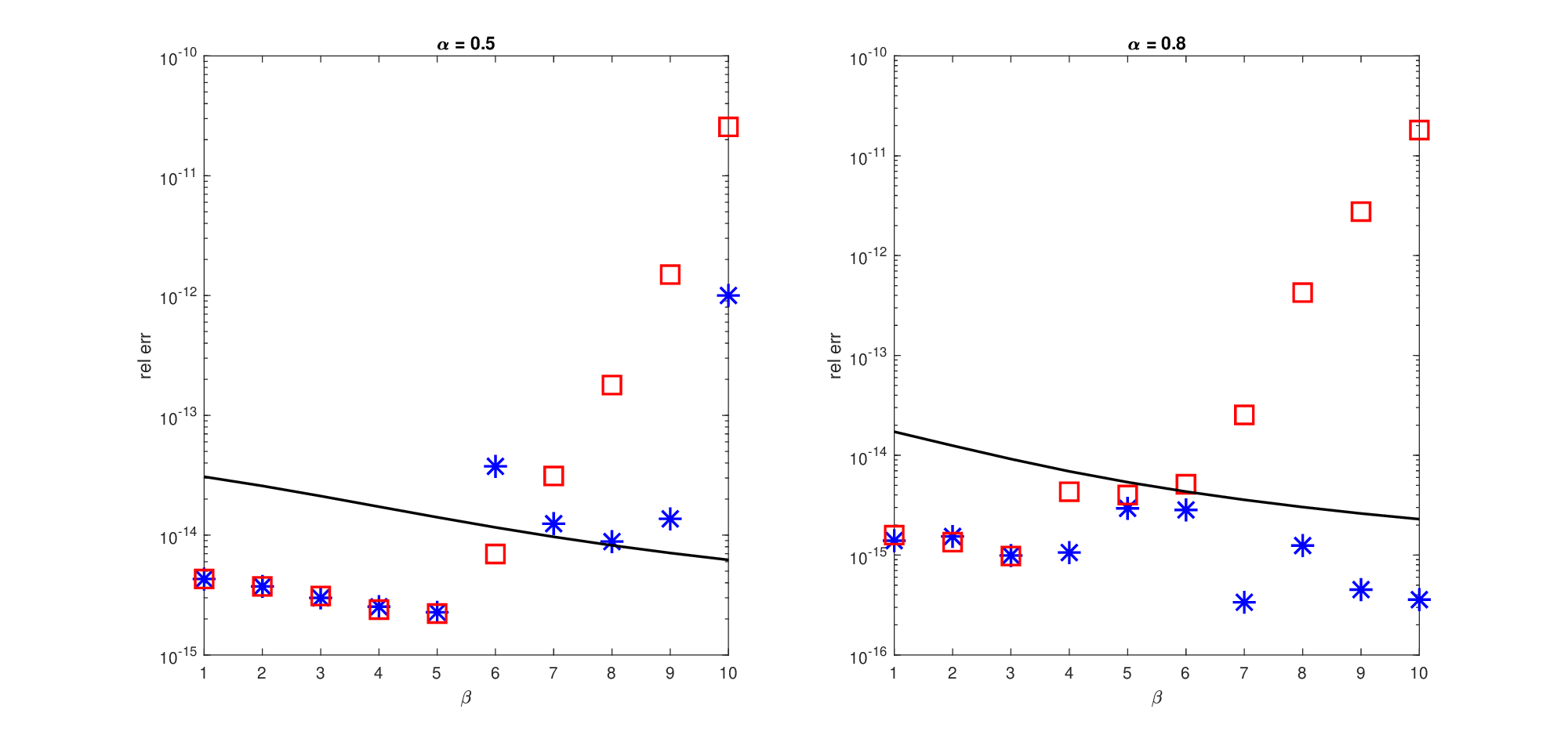}
		\caption{\small Results of Experiment 1. Relative errors of Algorithm \ref{algorithm3} (blue $\ast$) and of \texttt{GP-method} (red squares) for $\alpha=0.5$ and $\beta=1,\ldots,10$ (left) and $\alpha=0.8$ and $\beta=1,\ldots,10$ (right). The solid black line concerns to $\cond_{ML}(A)u$. }
		\label{fig-6}
	\end{figure}
\end{center}

\begin{center}
	\begin{figure}[t]
		\centering
		\includegraphics[width=0.99\textwidth]{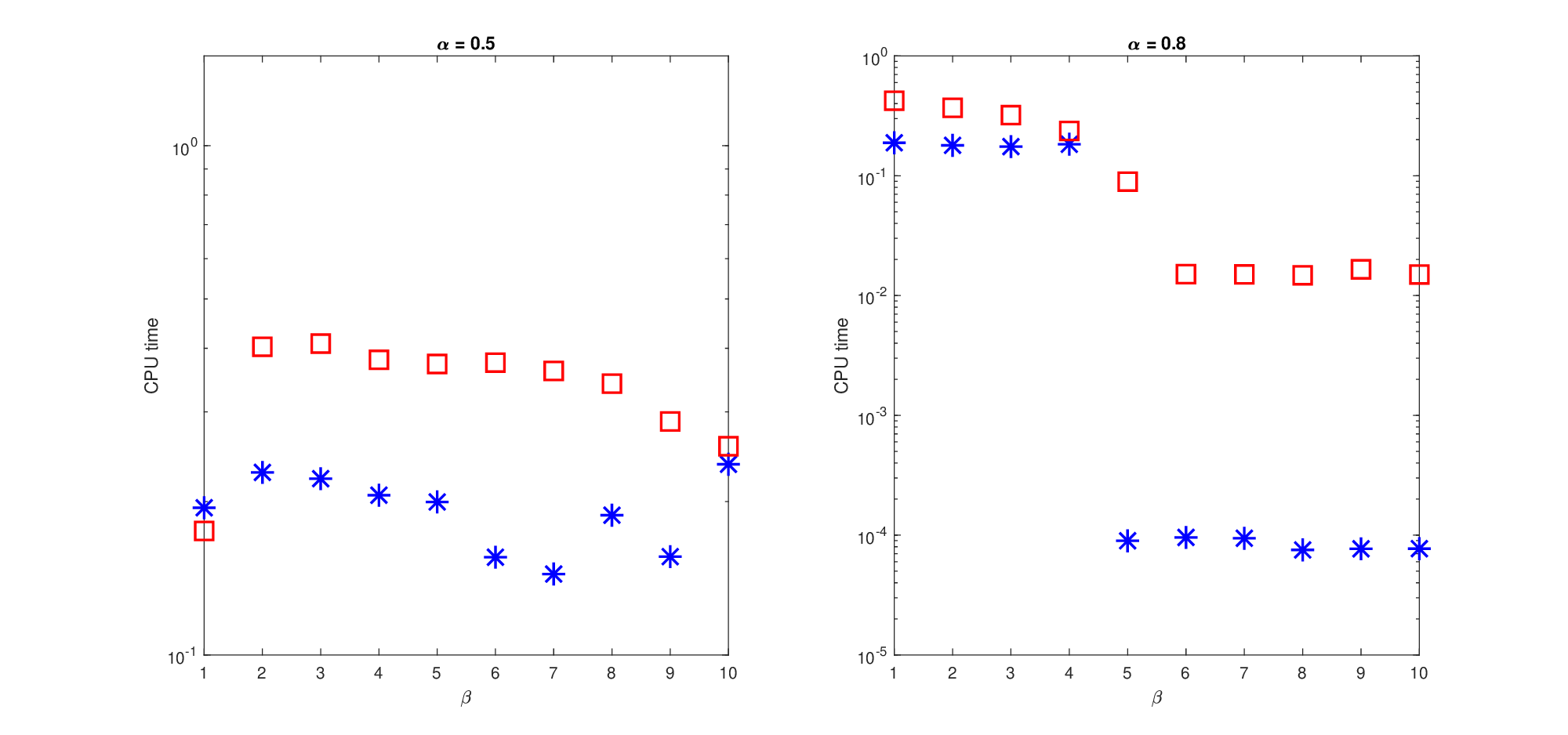}
		\caption{\small Results of Experiment 1. CPU time (seconds) of Algorithm \ref{algorithm3} (blue $\ast$) and of \texttt{GP-method} (red squares) for $\alpha=0.5$ and $\beta=1,\ldots,10$ (left) and $\alpha=0.8$ and $\beta=1,\ldots,10$ (right).  }
		\label{fig-6a}
	\end{figure}
\end{center}  

\medskip\noindent {\it Experiment 2.} This experiment is similar to the one in Example 2 of \cite{Garrappa18}. It involves four $40\times 40$ matrices with prescribed eigenvalues, with moderate multiplicities. The value of $\beta=1$ is fixed, while the value of $\alpha=0.6:0.4:2.6$ varies. The eigenvalues and the multiplicities of the matrices are included in Table \ref{tabela}. The relative errors for Algorithm \ref{algorithm3} and \texttt{GP-method}, as well as the estimates for $\cond_{ML}(A)u$, are plotted in Figure \ref{fig-7}. In Figure \ref{fig-7a} we can see the CPU time for this experiment. 

\begin{table}[h]
	\centering
		\begin{tabular}{ll}
		\hline
		 &  Eigenvalues (multiplicities)\\   \hline
		Matrix $A_1$ &   $\pm 1.0(5)\quad \pm 1.0001(4)\quad \pm 1.001(4)\quad \pm 1.01(4)\quad \pm 1.1(3)$\\
		Matrix $A_2$ & $\pm 1.0(8)\quad 2(8)\quad -5(8)\quad -10(8)$\\
		Matrix $A_3$ & $-1(2)\quad -5(2)\quad 1 \pm 10i(6)\quad -4 \pm 1.5i(6)\quad \pm 5i(6)$ \\
		Matrix $A_4$ & $1(4)\quad 1.0001(4)\quad 1.001(4)\quad 1 \pm 10i(7)\quad -4 \pm 1.5i(7)$\\
		\hline
	\end{tabular}
\caption{Eigenvalues and their multiplicities for the matrices used in Experiment 2} \label{tabela}
\end{table}
To design the matrices $A_1,\ldots,A_4$, we started by forming diagonal matrices having the assigned eigenvalues as entries, repeated according to the
multiplicities, and then applied a similarity transformation with randomized orthogonal matrices. These matrices have non trivial blocks in the blocked Schur form with reordering. The results are displayed in Figure \ref{fig-7}, where we observe that Algorithm \ref{algorithm3} and \text{GP-method} produce relative errors of similar magnitude, except in the cases when Algorithm \ref{algorithm3} calls the Taylor series, whose results are better. The deviation between the relative errors and the values of $\cond_{ML}(A)u$ is moderate. We stress once more that we are using the standard definition of the relative error, while the methods of \cite{Garrappa15,Garrappa18} do not.

\begin{center}
	\begin{figure}[t]
		\centering
		\includegraphics[width=0.99\textwidth]{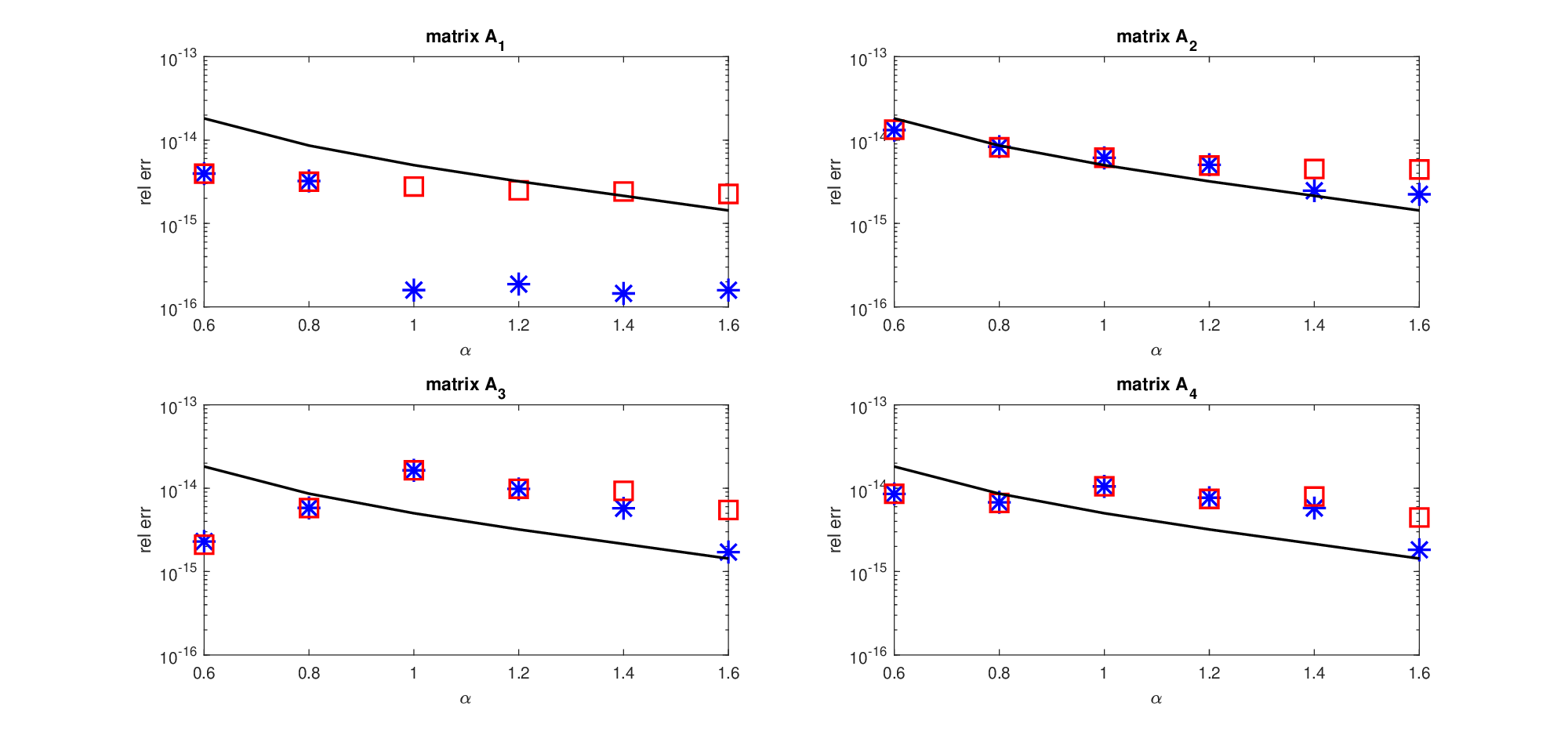}
		\caption{\small Results of Experiment 2. Relative errors of Algorithm \ref{algorithm3} (blue $\ast$) and of \texttt{GP-method} (red squares) with $\alpha=0.6:0.4:2.6$ and $\beta=1$. The solid black line represents estimates to $\cond_{ML}(A)u$. }
		\label{fig-7}
	\end{figure}
\end{center}  

\begin{center}
	\begin{figure}[t]
		\centering
		\includegraphics[width=0.99\textwidth]{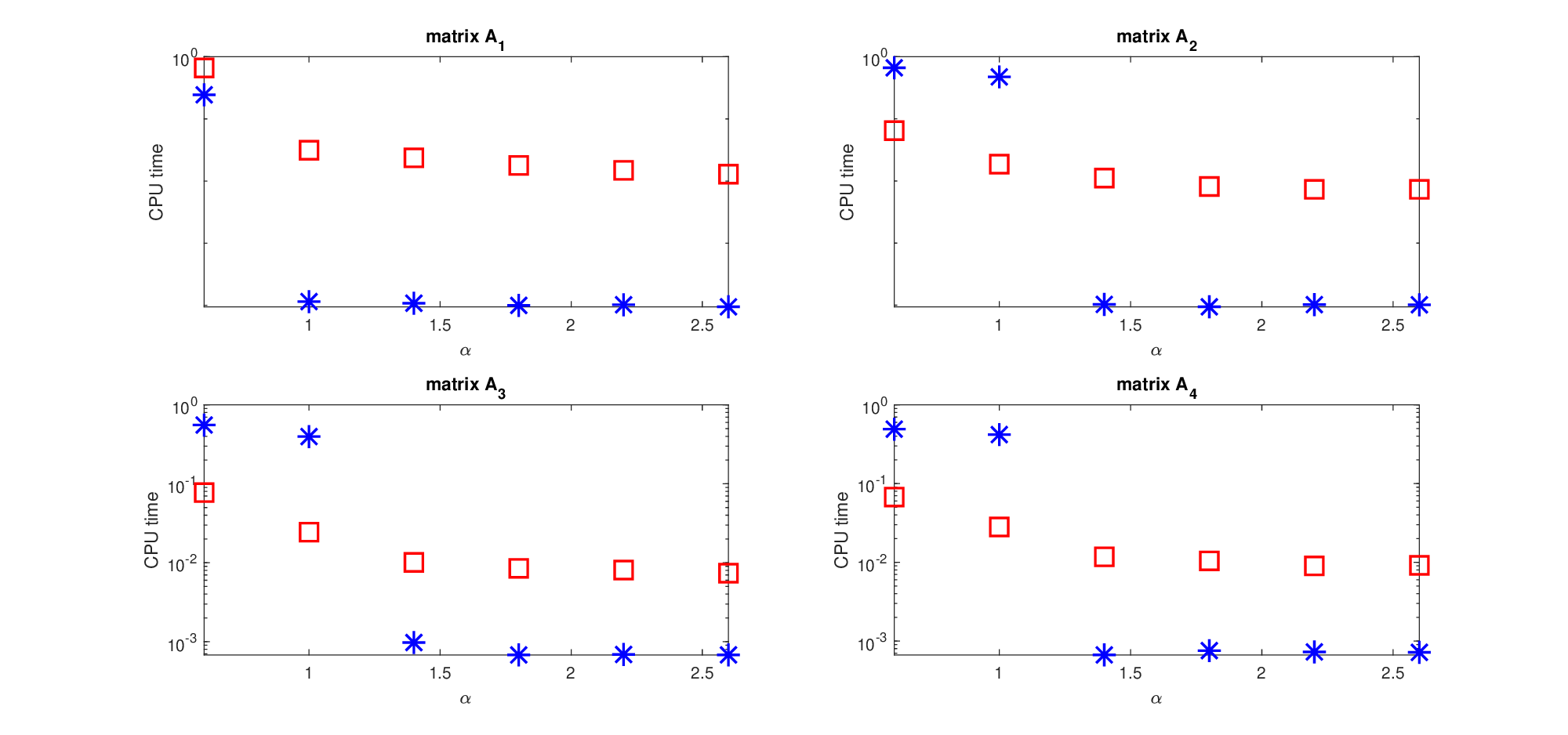}
		\caption{\small Results of Experiment 2. CPU time (seconds) of Algorithm \ref{algorithm3} (blue $\ast$) and of \texttt{GP-method} (red squares) with $\alpha=0.6:0.4:2.6$ and $\beta=1$.  }
		\label{fig-7a}
	\end{figure}
\end{center}  

\medskip\noindent {\it Experiment 3.} In this experiment, we consider $\alpha=0.8$ and $\beta=2$. It involves $15$ matrices from MATLAB's gallery, all of them with order $30$ and small blocks in the blocked Schur decomposition. Figure \ref{fig-8} compares the relative errors for Algorithm \ref{algorithm3} and \texttt{GP-method} with $\cond_{ML}(A)u$, while Figure \ref{fig-8a} displays the corresponding CPU time. All of the matrices were carefully chosen to avoid the occurrence of  overflow/underflow in the ML function.  

\begin{center}
	\begin{figure}[t]
		\centering
		\includegraphics[width=0.99\textwidth]{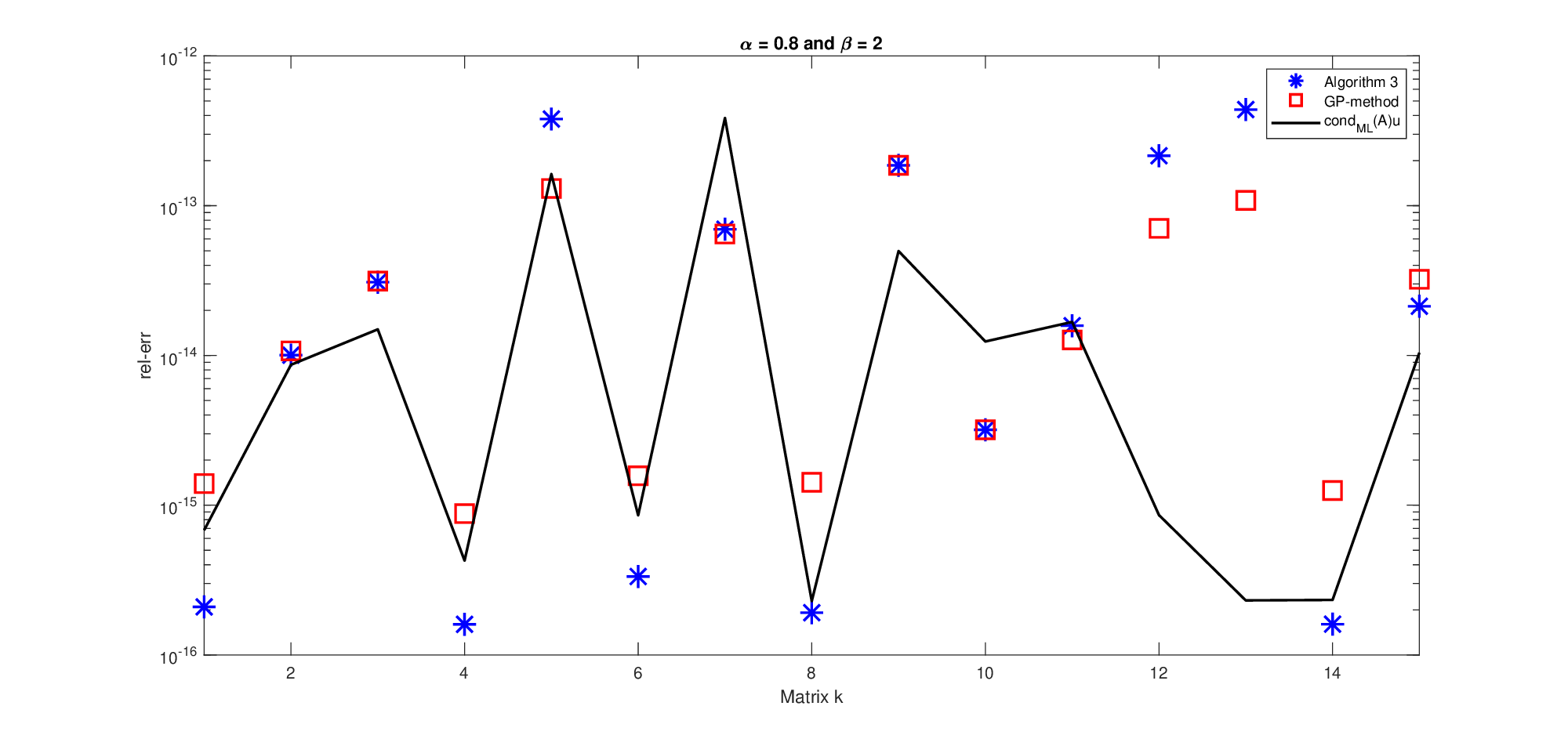}
		\caption{\small Results of Experiment 3. Relative errors of Algorithm \ref{algorithm3} (blue $\ast$) and of \texttt{GP-method} (red squares) with $\alpha=0.8$ and $\beta=2$ for fifteen matrices of order $30$ taken from MATLAB's gallery. The solid black line represents $\cond_{ML}(A)u$.  }
		\label{fig-8}
	\end{figure}
\end{center} 

\begin{center}
	\begin{figure}[t]
		\centering
		\includegraphics[width=0.99\textwidth]{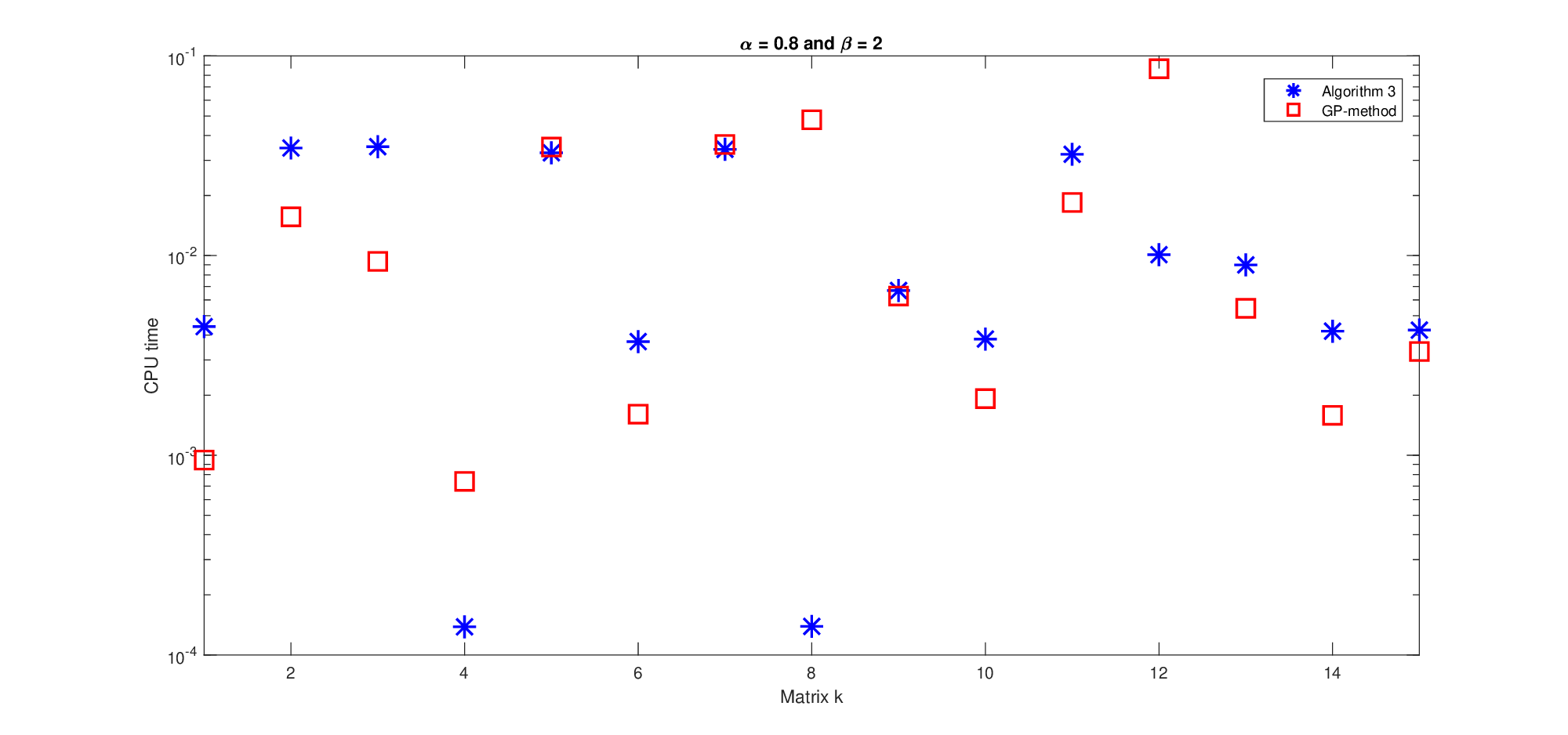}
		\caption{\small Results of Experiment 3. CPU time (in seconds) of Algorithm \ref{algorithm3} (blue $\ast$) and of \texttt{GP-method} (red squares) with $\alpha=0.8$ and $\beta=2$ for the fifteen matrices considered in Experiment 3. }
		\label{fig-8a}
	\end{figure}
\end{center} 

\medskip\noindent {\it Experiment 4.} This experiment involves sixteen $40\times 40$ matrices, where the first eight are Jordan blocks and the last eight are randomized atomic blocks with non-zero entries (except, of course, the ones in the lower part of the matrix). The results for the relative error are displayed in Figure \ref{fig-9}, while the CPU time is in Figure \ref{fig-9a}. This experiment shows clearly the advantage of using Algorithm \ref{algorithm3} in matrices with large atomic blocks, where Algorithm \ref{algorithm3} is about $100$ times faster. Moreover, the relative errors are smaller for almost all the test matrices.

\begin{center}
	\begin{figure}[t]
		\centering
		\includegraphics[width=0.99\textwidth]{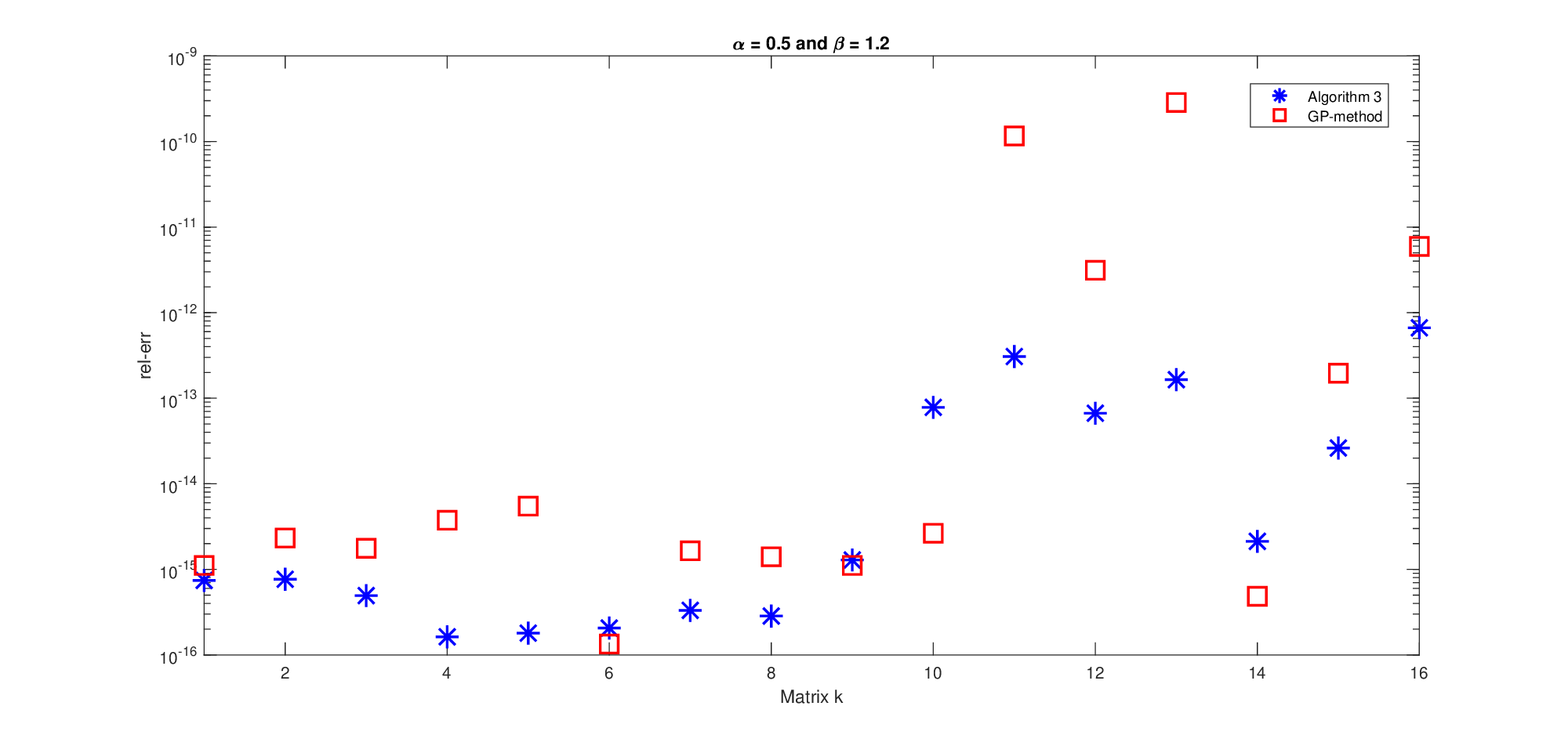}
		\caption{\small Results of Experiment 4. Relative errors of Algorithm \ref{algorithm3} (blue $\ast$) and of \texttt{GP-method} (red squares) with $\alpha=0.5$ and $\beta=1.2$ for sixteen atomic blocks of order $40$.   }
		\label{fig-9}
	\end{figure}
\end{center} 

\begin{center}
	\begin{figure}[t]
		\centering
		\includegraphics[width=0.99\textwidth]{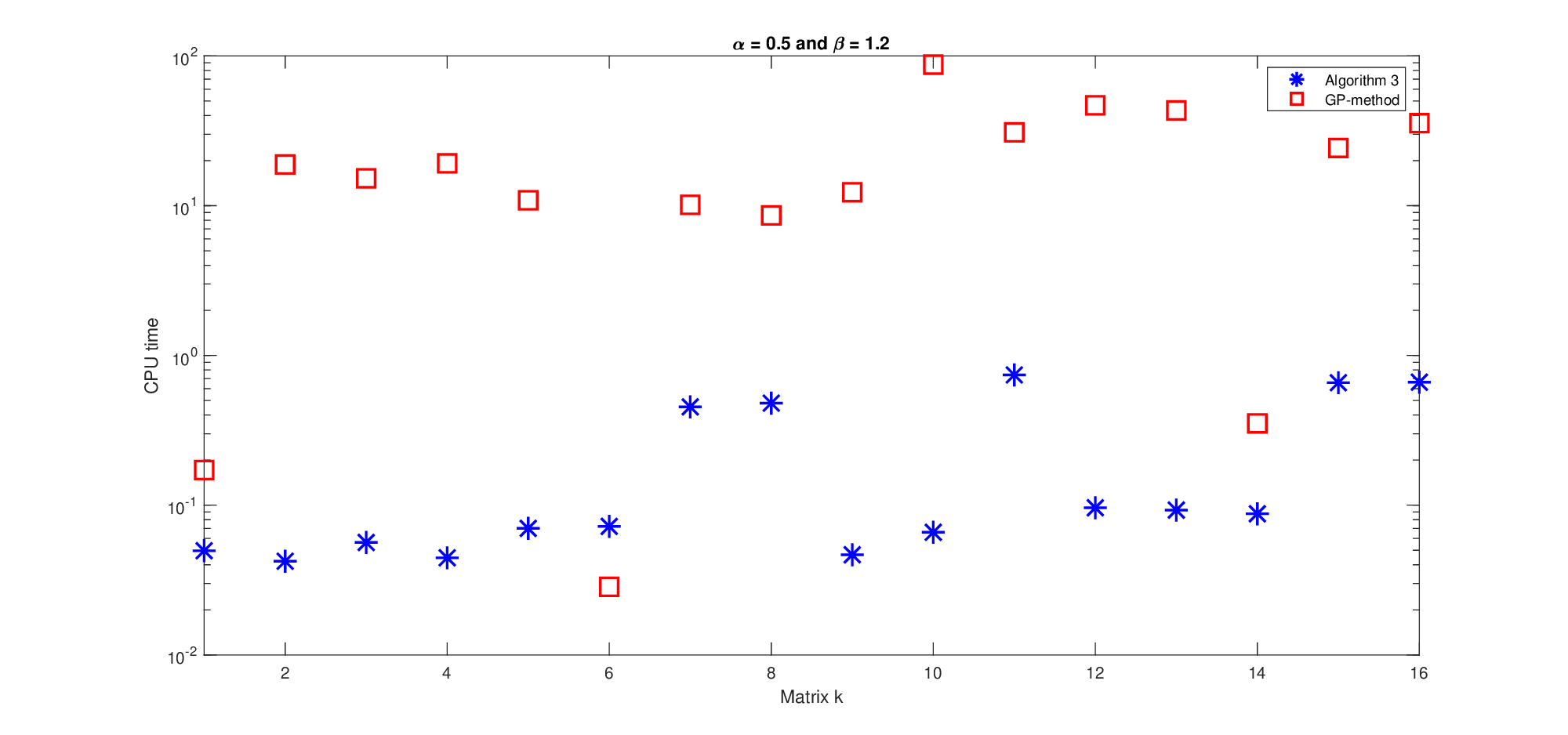}
		\caption{\small Results of Experiment 4. CPU time (in seconds) of Algorithm \ref{algorithm3} (blue $\ast$) and of \texttt{GP-method} (red squares) with $\alpha=0.5$ and $\beta=1.2$ for the sixteen sixteen atomic blocks of order $40$. }
		\label{fig-9a}
	\end{figure}
\end{center}

Now we shall make an overall comment on the results of the four experiments. When Algorithm \ref{algorithm3} calls the Taylor series for the computations (which depends on the the values of $\|A\|$, $\alpha$ and $\beta$), our tests have shown that it performs much better than \texttt{GP-method}. The same occurs when Algorithm \ref{algorithm3} uses the Schur-Parlett approach combined with the Cauchy integral formula for matrices with large atomic blocks. For the remaining cases, we can say that both methods have a similar behavior. We must recall, however, that Algorithm \ref{algorithm3} has the important advantage of being derivative free, so that it is much easier to extend to other matrix functions. We recall that the \texttt{GP-method} requires an efficient method for evaluating the higher order derivatives of the ML function. If extended to other matrix functions, a specific efficient method for evaluating their higher order derivatives would be required.  

\section{Conclusions}\label{sec-conclusions}

We have proved that if the values of $\|A\|$, $\alpha$ and $\beta$ satisfy certain conditions, the Taylor series can provide a very efficient method for evaluating the ML function (even when $\|A\|\geq 1$), that has the advantage of being derivative and transformation-free. For the situations where reliable results are not guaranteed by the Taylor method, we have shown that computing the ML function of the atomic blocks arising in the blocked Schur decomposition with reordering by means of the Cauchy integral formula provides a promising technique, if a suitable contour is chosen. This latter technique is derivative-free as well, but not transformation-free. A set of numerical experiments confirm that our overall algorithm combining the Taylor series with the Cauchy integral formula is competitive with the state-of-the-art method for IEEE double precision environments, and in some cases even better in terms of accuracy and CPU time.


\begin{thebibliography}{xx}
	
\bibitem{Agarwal} R.P. Agarwal, A propos d'une note de M. Pierre Humbert, C. R. Acad. Sci. Paris 236, 2031--2032, 1953.

\bibitem{Cardoso} J. Cardoso, Computation of the matrix $p$th root and its Fréchet derivative by integrals, Electron. Trans. Numer. Anal., 39, 414--436, 2012.

\bibitem{Davies03} P. Davies, N. Higham, A Schur--Parlett algorithm for computing matrix functions, SIAM J. Matrix Anal. Appl.
25, 464--485, 2003.

\bibitem{Davies05} P. Davies, N. Higham, Computing $f(A)b$ for matrix functions $f$, in QCD and numerical analysis III, 47, Lect. Notes Comput. Sci. Eng., 15--24, Springer, Berlin, 2005.

\bibitem{Davis84} P. Davis, P. Rabinowitz, {\em Methods of Numerical Integration}, 2nd ed., Academic Press, London,
1984.
	
\bibitem{Erdelyi} A. Erd\'{e}lyi, W. Magnus, F. Oberhettinger, F.G. Tricomi, {\em Higher Transcendental
Functions}, vol. 3. McGraw-Hill, New York, 1955.

\bibitem{Fasi} M. Fasi, Optimality of the Paterson–Stockmeyer method for evaluating matrix polynomials and rational matrix functions, Linear Algebra Appl., 574, 182--200, 2019.

\bibitem{Garrappa11} R. Garrappa, M. Popolizio, On the use of matrix functions for fractional partial differential equations,
Math. Comput. Simul. 81(5), 1045--1056, 2011.

\bibitem{Garrappa15} R. Garrappa, Numerical evaluation of two and three parameter Mittag-Leffler function. SIAM J. Numer. Anal. 53, 135--169, 2015.

\bibitem{Garrappa-m} R. Garrappa, The Mittag-Leffler function (https://www.mathworks.com/matlabcentral/fileexchange/48154-the-mittag-leffler-function), MATLAB Central File Exchange. Retrieved July 4, 2023.

\bibitem{Garrappa17} R. Garrappa, I. Moret, M. Popolizio, On the time-fractional Schr\"odinger equation: theoretical analysis
and numerical solution by matrix Mittag-Leffler functions, Comput. Math. Appl. 74(5), 977--992, 2017.

\bibitem{Garrappa18} R. Garrappa, M. Popolizio, Computing the matrix Mittag-Leffler function with
applications to fractional calculus, J. Sci. Comput., 77, 129--153, 2018.

\bibitem{Golub} G. Golub and C. Van Loan, {\em Matrix Computations}, 4th ed., Johns Hopkins University
Press, Baltimore, MD, 2013.

\bibitem{Gorenflo14} R. Gorenflo, A. Kilbas, F. Mainardi, S. Rogosin, {\em Mittag-Leffler Functions, Related Topics and Applications}, Springer-Verlag, 2014.

\bibitem{Gorenflo02} R. Gorenflo, J. Loutchko, Y. Luchko, Computation of the Mittag-Leffler function $E_{\alpha,\beta} (z)$ and its derivatives, 
Fract. Calc. Appl. Anal. 5 (4), 491-518, 2002. Erratum: Fract. Calc. Appl. Anal. 6, 2003. 

\bibitem{Higham08}
N. J. Higham, {\em Functions of Matrices: Theory and Computation.} Society for Industrial and Applied Mathematics, Philadelphia (2008).

\bibitem{Higham21} N. J. Higham, X. Liu, A Multiprecision derivative-free Schur-Parlett algorithm for computing matrix functions, SIAM J. Matrix Anal. Appl., 42(3), 1401--1422, 2021. 

\bibitem{mftoolbox} \textsc{N. Higham}, {\em The Matrix Function Toolbox}, 
(https://www.mathworks.com/matlabcentral/fileexchange/20820-the-matrix-function-toolbox), MATLAB Central File Exchange. Retrieved August 7, 2023.

\bibitem{Horn94} R. A. Horn, C. R. Johnson, {\em Topics in Matrix Analysis}, Cambridge Univ. Press, Cambridge, Paperback Edition (1994).

\bibitem{Humbert} P. Humbert, R.P. Agarwal, Sur la fonction de Mittag-Leffler et quelquenes de ses
g\'en\'eralisationes, Bull. Sci. Math. (Ser. II). 77, 180--185, 1953.

\bibitem{Mittag1} M.G. Mittag-Leffler, Sur l'int\'egrale de Laplace--Abel, Comp. Rend. Acad. Sci. Paris 135, 937--939, 1902.

\bibitem{Mittag2} M.G. Mittag-Leffler, Une g\'en\'eralization de l'int\'egrale de Laplace--Abel. Comp. Rend. Acad. Sci. Paris 136, 537?539, 1903.

\bibitem{Mittag3} M.G. Mittag-Leffler, Sur la nouvelle fonction $E_\alpha(x)$, Comp. Rend.  Acad. Sci. Paris 137, 554--558, 1903.

\bibitem{Mittag4} M.G. Mittag-Leffler, Sopra la funzione $E_\alpha(x)$, Rend. R. Acc. Lincei, (Ser. 5) 13, 3--5, 1904.

\bibitem{Mittag5} M.G. Mittag-Leffler, Sur la representation analytique d'une branche uniforme
d'une fonction monog\`ene (cinqui\`eme note), Acta Math. 29, 101--181, 1905.

\bibitem{Moreti11} I. Moret, P. Novati, On the convergence of Krylov subspace methods for matrix Mittag-Leffler functions, 
SIAM J. Numer. Anal. 49(5), 2144--2164, 2011.


\bibitem{Parlett} B. Parlett, A recurrence among the elements
of functions of triangular matrices, Linear Algebra Appl., 14,  117--121, 1976.

\bibitem{Paterson} M. Paterson and L. Stockmeyer, On the number of nonscalar multiplications necessary to evaluate polynomials, SIAM J. Comput., 2(1), 60–66, 1973.

\bibitem{Podlubny1} I. Podlubny, Fractional Differential Equations, {\em An Introduction to Fractional Derivatives,
Fractional Differential Equations, to Methods of Their Solution and Some of Their Applications}, Math. Sci. Engrg. 198, Academic Press Inc., San Diego, CA, 1999.

\bibitem{Podlubny2} I. Podlubny, Mittag-Leffler function (https://www.mathworks.com/matlabcentral/fileexchange/8738-mittag-leffler-function), MATLAB Central File Exchange. Retrieved July 4, 2023.


\bibitem{Seybold08} H. Seybold, R. Hilfer, Numerical algorithm for calculating the generalized Mittag-Leffler function, SIAM J. Numer. Anal., 47(1), 69--88, 2008.

\bibitem{Tatsuoka} F. Tatsuoka, T. Sogabe, Y. Miyatake, T. Kemmochi, 
S.-L. Zhang, Computing the matrix fractional power with the double exponential formula, Electron. Trans. Numer. Anal., 54, 558--580, 2021.

\bibitem{Trefethen} L. N. Trefethen, J. C. Weideman, The exponentially convergent
trapezoidal rule, SIAM Rev., Vol. 56 (3), 385--458, 2014.

\bibitem{Wiman} A. Wiman, \"{U}ber den fundamentalsatz der theorie der funkntionen $E_\alpha(x)$, Acta Math. 29, 191--201, 1905.


\end{thebibliography}
\end{document}